\documentclass[11pt,draft]{amsart}

\usepackage{amssymb}

	\normalbaselines						  %

\newcommand{\R}{{\mathbb  R}}

\newcommand{\Z}{{\mathbb  Z}}
\newcommand{\N}{{\mathbb  N}}
\newcommand{\C}{{\mathbb  C}}

\newcommand{\dd}{{d}}
\newcommand{\ID}{{\mathbf{1}}}

\newcommand{\OID}{{\mathbf{I}}}

\newcommand{\fdot}{\,\cdot\,}

\newcommand{\wt}{\widetilde}

\newcommand{\cH}{\mathcal{H}}

\newcommand{\f}{\varphi}

\newcommand{\e}{\varepsilon}

\DeclareMathOperator{\clos}{clos}
\DeclareMathOperator{\conv}{conv}
\DeclareMathOperator{\wclos}{w-clos}

\DeclareMathOperator{\Ker}{Ker}
\DeclareMathOperator{\Ran}{Ran}
\DeclareMathOperator{\spa}{span}
\DeclareMathOperator{\im}{Im}
\DeclareMathOperator{\re}{Re}
\DeclareMathOperator{\supp}{supp}
\DeclareMathOperator{\dist}{dist}

\newcommand{\la}{\lambda}

\newcommand{\ci}[1]{_{ {}_{\scriptstyle #1}}}
\newcommand{\ti}[1]{_{\scriptstyle \text{\rm #1}}}

\catcode`@=11
\chardef\mathlig@atcode\count255

\def\actively#1#2{\begingroup\uccode`\~=`#2\relax\uppercase{\endgroup#1~}}
\def\mathlig@gobble{\afterassignment\mathlig@next@cmd\let\mathlig@next= }
\def\mathlig@delim{\mathlig@delim}
\def\mathlig@defcs#1{\expandafter\def\csname#1\endcsname}
\def\mathlig@let@cs#1#2{\expandafter\let\expandafter#1\csname#2\endcsname}
\def\mathlig@appendcs#1#2{\expandafter\edef\csname#1\endcsname{\csname#1\endcsname#2}}

\def\mathlig#1#2{\mathlig@checklig#1\mathlig@end\mathlig@defcs{mathlig@back@#1}{#2}\ignorespaces}

\def\mathlig@checklig#1#2\mathlig@end{%
 \expandafter\ifx\csname mathlig@forw@#1\endcsname\relax
 \expandafter\mathchardef\csname mathlig@back@#1\endcsname=\mathcode`#1%
 \mathcode`#1"8000\actively\def#1{\csname mathlig@look@#1\endcsname}%
 \mathlig@dolig#1\mathlig@delim
\fi
\mathlig@checksuffix#1#2\mathlig@end
}

\def\mathlig@checksuffix#1#2\mathlig@end{%
\ifx\mathlig@delim#2\mathlig@delim\relax\else\mathlig@checksuffix@{#1}#2\mathlig@end\fi
}
\def\mathlig@checksuffix@#1#2#3\mathlig@end{%
\expandafter\ifx\csname mathlig@forw@#1#2\endcsname\relax\mathlig@dosuffix{#1}{#2}\fi
\mathlig@checksuffix{#1#2}#3\mathlig@end
}

\def\mathlig@dosuffix#1#2{%
\mathlig@appendcs{mathlig@toks@#1}{#2}%
\mathlig@dolig{#1}{#2}\mathlig@delim
}

\def\mathlig@dolig#1#2\mathlig@delim{%
 \mathlig@defcs{mathlig@look@#1#2}{%
 \mathlig@let@cs\mathlig@next{mathlig@forw@#1#2}\futurelet\mathlig@next@tok\mathlig@next}%
 \mathlig@defcs{mathlig@forw@#1#2}{%
  \mathlig@let@cs\mathlig@next{mathlig@back@#1#2}%
  \mathlig@let@cs\checker{mathlig@chck@#1#2}%
  \mathlig@let@cs\mathligtoks{mathlig@toks@#1#2}%
  \expandafter\ifx\expandafter\mathlig@delim\mathligtoks\mathlig@delim\relax\else
  \expandafter\checker\mathligtoks\mathlig@delim\fi
  \mathlig@next
 }%
 \mathlig@defcs{mathlig@toks@#1#2}{}%
 \mathlig@defcs{mathlig@chck@#1#2}##1##2\mathlig@delim{%
  \ifx\mathlig@next@tok##1%
   \mathlig@let@cs\mathlig@next@cmd{mathlig@look@#1#2##1}\let\mathlig@next\mathlig@gobble
  \fi
  \ifx\mathlig@delim##2\mathlig@delim\relax\else
   \csname mathlig@chck@#1#2\endcsname##2\mathlig@delim
  \fi
 }%
 \ifx\mathlig@delim#2\mathlig@delim\else
  \mathlig@defcs{mathlig@back@#1#2}{\csname mathlig@back@#1\endcsname #2}%
 \fi
}%

\catcode`@\mathlig@atcode

\mathchardef\ordinarycolon\mathcode`\:
\def\vcentcolon{\mathrel{\mathop\ordinarycolon}}
\mathlig{:=}{\vcentcolon=}
\mathlig{::=}{\vcentcolon\vcentcolon=}

\numberwithin{equation}{section}

\theoremstyle{plain}
\newtheorem{theo}{Theorem}[section]
\newtheorem{cor}[theo]{Corollary}
\newtheorem{lem}[theo]{Lemma}
\newtheorem{prop}[theo]{Proposition}

\theoremstyle{definition}

\theoremstyle{remark}
\newtheorem*{ex*}{Example}
\theoremstyle{remark}
\newtheorem*{exs*}{Examples}
\theoremstyle{remark}
\newtheorem*{rems*}{Remarks}
\theoremstyle{remark}
\newtheorem*{rem*}{Remark}

\title[Rank one perturbations]{Rank one perturbations and singular integral operators}
\author{Constanze~Liaw}
\address{Department of Mathematics, Brown University, 151 Thayer 
Str./Box 1917,      
 Providence, RI  02912, USA }
\email{conni@math.brown.edu}
\author{Sergei~Treil}
\thanks{The second author is  partially supported by the NSF grant DMS-0800876.
}
\address{Department of Mathematics, Brown University, 151 Thayer 
Str./Box 1917,      
 Providence, RI  02912, USA }
\email{treil@math.brown.edu}
\urladdr{http://www.math.brown.edu/\~{}treil}

\begin{document}

\begin{abstract}
We consider  rank one perturbations $A_\alpha=A+\alpha(\fdot,\varphi)\varphi$ of a self-adjoint operator $A$ with cyclic vector $\varphi\in\mathcal H_{-1}(A)$ on a Hilbert space $\mathcal H$. The spectral representation of the perturbed operator $A_\alpha$ is given by a singular integral operator of special form. Such operators exhibit what we call 'rigidity' and are connected with two weight estimates for the Hilbert transform.

Also, some results about two weight estimates of Cauchy (Hilbert) transforms are proved. In particular, it is proved that the regularized Cauchy transforms $T_\e$ are uniformly (in $\e$) bounded operators from $L^2(\mu)$ to $L^2(\mu_\alpha)$, where $\mu$ and $\mu_\alpha$ are the spectral measures of $A$ and $A_\alpha$, respectively. 

As an application, a sufficient condition for $A_\alpha$ to have a pure absolutely continuous spectrum on a closed interval is given in terms of the density of the spectral measure of $A$ with respect to $\varphi$. Some examples, like Jacobi matrices and Schr\"odinger operators with $L^2$ potentials are considered.

\end{abstract}

\maketitle

\

\


\section{Introduction}
\subsection{Setup of rank one perturbations}\label{setup}
Let $A$ be a self-adjoint (possibly unbounded)  operator on a Hilbert space $\mathcal H$. We are considering a family of rank-one perturbations $A+\alpha(\,\cdot\,,\f)\f$. Here, if the operator $A$ is bounded, $\f$ is a vector in $\mathcal H$. For unbounded $A$, we consider the wider class of so-called form bounded perturbations where we assume $\varphi\in\mathcal H_{-1}(A)\supset \mathcal H$, so the perturbation $\alpha (\fdot, \f)\f$ can be unbounded (see subsection \ref{SMAIN} below for definition).

It is possible that the results of the paper hold for a wider class of perturbations than form bounded, but we restricted ourselves  to avoid problems defining the perturbation, which can be non-unique.

Without loss of generality, we can assume that $A$  has simple spectrum and that  $\varphi$ is a cyclic vector for $A$, i.e.~that the linear span of $\{(A-z\OID)^{-1}\varphi:z\in\C\}$ is dense in $\mathcal H$. According to the Spectral Theorem, $A$ is unitary equivalent to a multiplication operator $M_t:f(t)\mapsto tf(t)$ on $ L^2(\mu)$ for some (non-unique) Borel measure $\mu$. We make the spectral measure unique by letting $\mu$ be the spectral measure corresponding to $\varphi$, i.e.~$\mu:=\mu^\varphi$, where $\mu^\f$ is the unique measure such that
\begin{align*}
&&\int\ci\R \frac{1}{t-\lambda}\,\dd\mu^\varphi(t)=((A-\lambda\OID)^{-1}\varphi,\varphi)\ci{\mathcal H}
\qquad\forall\lambda\in\C\setminus\sigma(A).
\end{align*}
Existence and uniqueness of such $\mu$ is guaranteed by the Spectral Theorem.

It is easy to see that in this representation vector $\varphi$ is represented by the function $\ID$, meaning that if $U: \cH\to L^2(\mu)$ is the unitary operator such that $M_t = U A U^{-1}$, then $U\f= \ID$.  
As will be explained later in section \ref{SMAIN}, in this representation the assumption $\f\in \mathcal{H}_{-1}(A)$ means simply that $\int\ci\R (1+|t|)^{-1} d\mu(t)<\infty$.

Without loss of generality, assume henceforth that $A=M_t$ on $ L^2(\mu)$, $\int\ci\R (1+|t|)^{-1} d\mu(t)<\infty$, and $\varphi\equiv\ID$.
Consider the family of self-adjoint rank one perturbations
\[
A_\alpha := A +\alpha(\,\cdot\,,\f)\f\qquad\forall \alpha \in\R.
\]
In the case of form bounded perturbations this formal definition of $A_\alpha$ can be made precise, see e.g.~\cite{AK1}.

\begin{rem*}
 By assuming simplicity of the spectrum, i.e.~ the existence of a cyclic vector $\varphi$ for $A$, we do not forfeit generality. Indeed, if there is no cyclic vector, we decompose $\mathcal H$ into an orthogonal sum of Hilbert spaces $\mathcal H=\widetilde{\mathcal H}\oplus\widehat{\mathcal H}$ such that $\varphi$ is cyclic for the restriction $A|\ci{\widetilde{\mathcal H}}$. Then for all $\alpha\in\R$ we have $A_\alpha|\ci{\widehat{\mathcal H}}=A|\ci{\widehat{\mathcal H}}$, and it suffices to investigate the behavior of $A_\alpha$ on $\widetilde{\mathcal H}$.
\end{rem*}

It is well-known that $\f$ is cyclic for operators $A_\alpha$ as well, so $A_\alpha$ are unitary equivalent to multiplication by the independent variable in the spaces $L^2(\mu_\alpha)$. For a proof of the cyclicity confer the proof of Theorem \ref{t-repr-V} below for bounded $A$ and Lemma \ref{cyc} below in the case of form bounded perturbations.

Without loss of generality, let us make the measure $\mu_\alpha$ unique by choosing $\mu_\alpha$ to be the spectral measure corresponding to the vector $\ID$ in each $L^2(\mu_\alpha)$. So $\f$ is represented by $\ID$ in each $L^2(\mu_\alpha)$.

\subsection{Notation}
We will use the symbol $t$ for the independent variable in $L^2(\mu)$ and $s$ for the independent variable in $L^2(\mu_\alpha)$, so $M_t$ and $M_s$ are the multiplication by the independent variable in $L^2(\mu)$ and $L^2(\mu_\alpha)$, respectively.  Slightly abusing notation we will use subscripts $t$ or $s$ to indicate whether we are treating a function as an element of $L^2(\mu)$ or $L^2(\mu_\alpha)$ for regular perturbations, or as a point in $\cH_{-1}(A)$ for singular form bounded perturbations. Thus $\ID_t$ means the function  $\f\equiv \ID$, treated as a point in $L^2(\mu)$, while $\ID_s$ stands for the same function considered to be an element of $L^2(\mu_\alpha)$.

\subsection{Outline}
In section \ref{PROOF1}, we obtain a formula for the spectral representation of the perturbed operator $A_\alpha$. As a partial converse of this representation theorem, we show a certain rigidity for such operators. That is, integral operators represented by such a formula are unitary up to certain scaling and give rise to a rank one perturbation setting.

In section \ref{SIO}, we concentrate on singular integral operators. By a standard approximation argument, we show that the spectral representation of $A_\alpha$ is a singular integral operator. We obtain an alternative formula for the spectral representation of $A_\alpha$. We prove that certain regularizations of the Hilbert transform are uniformly bounded from $L^2(\mu)$ to $L^2(\nu)$ under very weak conditions on the measures $\mu$ and $\nu$. In particular, we allow non-doubling measures.

As an application of the representation theorem and the statements on singular integral operators, we prove, in section \ref{absSEC}, two results about the absence of embedded singular spectrum in the rank one perturbation setting. 

In section \ref{REXA}, we present examples of rank one perturbations. In all examples, the unperturbed operator $A$ has arbitrary embedded singular spectrum which resolves completely as soon as we 'switch on' the perturbation. The unperturbed operators include Hilbert--Schmidt perturbations of the free Jacobi operator, as well as Schr\"odinger operators with $L^2$ potentials.

\section{Spectral representation of the perturbation $A_\alpha$ and its properties}\label{PROOF1}
As mentioned above, by the Spectral Theorem, operators $A_\alpha$ are unitary equivalent to the multiplication $M_s$ by the independent variable $s$ in the space $L^2(\mu_\alpha)$, i.e.~there exists a unitary operator $V_\alpha:L^2(\mu)\to L^2(\mu_\alpha)$ such that $V_\alpha A_\alpha = M_s V_\alpha$.

Operator $V_\alpha$ is the spectral representation of $A_\alpha$. The measure $\mu_\alpha$ contains all spectral information of $A_\alpha$. Indeed, it is shown below that $\ID_t$ is cyclic for $A_\alpha$.

Let us give an integral representation for this  unitary operator. Without loss of generality we assume that $A$ is the multiplication operator $M_t$ by the independent variable $t$ in $L^2(\mu)$, $A_\alpha= A + \alpha(\,\cdot\,,\f)\f$, $\f\equiv\ID_t$. We assume that $A_\alpha$ is a form bounded perturbation, i.e.~$\int (1+|t|)^{-1}d\mu(t) <\infty$. We consider $\mu_\alpha$ to be the spectral measure of $A_\alpha$ corresponding to $\ID_t$.

\begin{theo}[Representation Theorem]\label{t-repr-V}
Assume the above assumptions. The spectral representation $V_\alpha: L^2(\mu)\to L^2(\mu_\alpha)$ of $A_\alpha$ is given by
\begin{equation}
\label{repr-V}
V_\alpha f(s) = f(s) -\alpha \int\frac{f(s)-f(t)}{s-t}\,d\mu(t)
\end{equation}
for all compactly supported $C^1$ functions $f$.
\end{theo} 


Integral operators represented by formula \eqref{repr-V} are very interesting objects, probably deserving more careful investigation. Let us mention one property, which can be understood as a converse to the latter representation theorem.

\begin{theo}[Rigidity Theorem]
\label{rigTM}
Let measure $\mu$ on $\R$ be supported on at least two distinct points and satisfy $\int(1+|t|)^{-1}\,d\mu(t) <\infty$. 
Let $V$ be defined on compactly supported $C^1$ functions $f$ by formula \eqref{repr-V}.

Assume $V$ extends to a bounded operator from $L^2(\mu)$ to $L^2(\nu)$ and assume $\Ker V=\{0\}$.

Then there exists a function $h$ such that $1/h\in L^\infty(\nu)$, and $M_h V$ is a unitary operator from $L^2(\mu)\to L^2(\nu)$ (equivalently, that $V:L^2(d\mu)\to L^2( |h|^{2}\,d\nu)$ is unitary). 
 
Moreover, the unitary operator $U:=M_hV$ gives the spectral representation of the operator $A_\alpha := M_t + \alpha (\fdot, \f)\f$, $\f\equiv \ID$, in $L^2(\mu)$, namely $UA_\alpha = M_s U$, where $M_s$ is the multiplication by the independent variable $s$ in $L^2(\nu)$. 
\end{theo}

Theorem \ref{rigTM} will be proved in subsection \ref{PROOFR} below.

\subsection{Proof of Theorem \ref{t-repr-V} for bounded $A$}\label{re1}
Assume the hypotheses of the representation theorem, Theorem \ref{t-repr-V}, and let $A$ be bounded. Recall that for bounded $A$ we have $\ID_t\in L^2(\mu)$, by assumption. In fact, bounded $A$ implies $\cH_{-1}(A)=\cH(A)=L^2(\mu)$, see section \ref{SMAIN} below.

Let us show that the vector $\ID_t$ is cyclic for $A_\alpha$. 

Recall that for a bounded operator $A=A^*$, the linear span of $\{ (A-\la \OID)^{-1} \f: \la\in\C \setminus \R\}$ is dense in $\cH$ if and only if the linear span of the orbit $\{A^n\f : n\ge 0\}$ is dense (in fact, the latter property is often used as the definition of a cyclic vector in the bounded case). 

Since $\mu$ is compactly supported, polynomials are dense in $L^2(\mu)$. It is easy to see that the functions $A_\alpha^n \f $, $\f= \ID_t$, are polynomials of degree exactly $n$.  Hence the linear span of $\{A_\alpha^n \ID_t\}\ci{n\in\N}$ is the set of all polynomials, and thus  is  dense in $L^2(\mu)$. So $\ID_t$ is cyclic for $A_\alpha$.

The identity 
\[
M_s V_\alpha =V_\alpha A_\alpha=V_\alpha [M_t+\alpha(\fdot,\ID_t)\ci{L^2(\mu)} \ID_t] 
\]
implies
\[
V_\alpha M_t=M_sV_\alpha -\alpha(\fdot,\ID_t)\ci{L^2(\mu)}V_\alpha \ID_t=M_sV_\alpha -\alpha(\fdot,\ID_t)\ci{L^2(\mu)} \ID_s.
\]

Using induction one can show that the identity 
$$
V_\alpha M_t^{n}=M_s^{n}V_\alpha -\alpha\sum_{k=0}^{n-1}(\fdot,a_k)\ci{L^2(\mu)} b_k,   
$$
where $a_k\in L^2(\mu),  a_k (t) = t^k$,  $b_k\in L^2(\mu_\alpha),    b_k(s) =s^{n-k-1}$, 
or, more informally,
\begin{equation}
\label{V-ind}
V_\alpha M_t^{n}=M_s^{n}V_\alpha -\alpha\sum_{k=0}^{n-1}(\fdot,t^{k})\ci{L^2(\mu)} s^{n-k-1}
\end{equation}
holds true for all $n\in\N$. 
Indeed, assuming that the above identity holds for $n-1$ we get
\begin{align*}
&V_\alpha M_t^n
=
V_\alpha M_tM_t^{n-1}=M_sV_\alpha M_t^{n-1}-\alpha(\fdot,t^{n-1})\ci{L^2(\mu)} \ID_s\\
=&M_s\left[M_s^{n-1}V_\alpha -\alpha \sum_{k=0}^{n-2}(\fdot,t^{k})\ci{L^2(\mu)} s^{n-k-2}\right]-\alpha (\fdot,t^{n-1})\ci{L^2(\mu)} \ID_s\\
=&M_s^nV_\alpha -\alpha \sum_{k=0}^{n-1}(\fdot,t^{k})\ci{L^2(\mu)} s^{n-k-1}.
\end{align*}

Since
$$
(f, t^{k})\ci{L^2(\mu)} s^{n-k-1} = \int\ci\R f(t)  t^k s^{n-k-1} \, d\mu(t)
$$
we have 
$$
\sum_{k=0}^{n-1} (\ID_t,t^{k})\ci{L^2(\mu)} s^{n-k-1} = \int\ci\R \left(  \sum_{k=0}^{n-1} t^k s^{n-k-1} \right) \,d\mu(t) = \int\ci\R \frac{s^n-t^n}{s-t}\,\dd\mu(t).
$$
Note, that the integral is well-defined, because $\mu(\R)<\infty$ and function $t\mapsto(s^n-t^n)/(s-t)$ is bounded on the (bounded) support of $\mu$.

So applying \eqref{V-ind} to $\ID_t\in L^2(\mu)$ and using the above identity we get 
\[
(V_\alpha t^n)(s) = s^n
-\alpha \int\ci\R \frac{s^n-t^n}{s-t}\,\,\dd\mu(t)
\]
for all $n\in\N$. Since $V_\alpha\ID_t=\ID_s$, this representation formula holds also on constant functions.

Due to the linearity of $V_\alpha $, this extends to a representation formula
\begin{align*}
(V_\alpha\, p )(s) =
 p(s)
-\alpha\int\ci\R \frac{p(s)-p(t)}{s-t}\dd\mu(t)
\end{align*}
on polynomials $p(t)$.

To extend this formula to $C_0^1(\R)$ we will use the lemma below. While in this case  a bit simpler direct reasoning is possible, the lemma below will be useful later, when we need to extend formula \eqref{repr-V} to different classes of functions. 

\begin{lem}
\label{lim-repr-V}
Let $\mu$ and $\nu$ be measures on $\R$ satisfying $\int(1+|x|)^{-1}\,d\mu(x)<\infty$, $\int(1+|x|)^{-2}\,d\nu(x)<\infty$.  Let $V:L^2(\mu) \to L^2(\nu)$ be a bounded operator such that for functions $f$ in some subset $\mathcal L\subset L^2(\mu)\cap L^2(\nu) \cap C^1(\R)$ we have
\begin{equation}
\label{repr-V-1}
Vf(s) = f(s) -\alpha \int\frac{f(s)-f(t)}{s-t}\,d\mu(t)\qquad \nu\text{-a.e.,}	
\end{equation}
where the integral is well-defined (integrand belongs to $L^1(\mu)$). 

Let $f_n\in \mathcal L$ be such that 
\begin{enumerate}
	\item\label{lrV-1} $f_n \to f$ $\mu$-a.e.~and $\nu$-a.e.;
	\item $|f_n(x)|\le C/(1+|x|)$ ($C$ does not depend on $n$);
	\item $|f_n'(x)|\le C$ ($C$ does not depend on $n$). 
\end{enumerate}

Then $f\in L^2(\mu)$ and $Vf$ is given by the above formula \eqref{repr-V-1} (note that we neither assumed nor concluded that $f\in\mathcal L$). 
\end{lem}

\begin{proof}
Assumptions (1) and (2) together with the assumptions about the measures and the Dominated Convergence Theorem imply that $f_n\to f$ in $L^2(\mu) $ and $L^2(\nu)$. The boundedness of $V$ implies that $Vf_n\to Vf$ in $L^2(\nu)$. By taking a subsequence, if necessary, we can always assume that $f_n\to f$, $Vf_n\to Vf$ with respect to $\nu$-a.e.

On the other hand by the Dominated Convergence Theorem for any fixed $s\in \R$ we have
$$
\lim_{n \to \infty} \int\frac{f_n(s)-f_n(t)}{s-t}\,d\mu(t) = \int\frac{f(s)-f(t)}{s-t}\,d\mu(t). 
$$
Indeed, we know that  $|f_n|\le C$, $|f_n'|\le C$. So for $|s-t|\le1$ it holds
$$
\frac{|f_n(s)-f_n(t)|}{|s-t|} \le C
$$
by the Mean Value Theorem. And for $|s-t|>1$ we have
\[
\frac{|f_n(s)-f_n(t)|}{|s-t|} \le \frac{2C}{|s-t|} . 
\]
Combining these two estimates, we get
\[
 \left| \frac{f_n(s)-f_n(t)}{s-t}\right|
 \le \frac{C(s)}{1+|t|}\,.
\]
Because $\int(1+|t|)^{-1}\,d\mu(t)<\infty$, we can apply the Dominated Convergence Theorem.
\end{proof}

To prove Theorem \ref{t-repr-V} in the general case, let us first remind the reader of a few well-known facts about form bounded perturbations. 

\subsection{Form bounded perturbations and resolvent formula}\label{SMAIN} 

For an unbounded self-adjoint operator $A$ in a Hilbert space $\cH$, one can define the standard scale of spaces 
\[
\hdots\subset\mathcal H_{2}(A)\subset\mathcal H_{1}(A)\subset\mathcal
H_{0}(A)= \cH \subset\mathcal H_{-1}(A)\subset\mathcal
H_{-2}(A)\subset\hdots,
\]
where $\mathcal H_r(A):=\left\{\psi\in
\cH :\left\|(1+|A|)^{r/2}\psi\right\|\ci{\cH}<\infty\right\}$ for
$r\geq0$. Here $|A|$ is the \emph{modulus} of the operator $A$, i.e.~$|A| = (A^*A)^{1/2}$. 

 If $r<0$, it is defined by $\mathcal H_r(A):=\left[\mathcal H_{-r}(A)\right]^*$ with the duality inherited from the inner product in $\cH$. Or, speaking more carefully, one can say that the space $\cH_{-r}$, $r>0$ is defined by introducing the norm 
$$
\| f\|\ci{\cH_{-r}} = \| (I +|A|)^{-r/2}f\|\ci{\cH}
$$
on $\cH$ and taking the completion of $\cH$ in this norm.

In the case when $A$ is the multiplication operator $M_t$ by the independent variable $t$ in $L^2(\mu)$, we simply have 
$$
\cH_{r} =L^2((1+|t|)^rd\mu) =  \left\{ f\,:\, \int |f(t)|^2 (1+|t|)^r \,d\mu(t) <\infty\right\}. 
$$
Note, if $A$ is a bounded operator, then $\cH_r=\cH$ for all $r$. 

It is well-known that it is possible to define the rank one perturbation $A_\alpha= A +\alpha(\fdot, \f)\f$ of the operator $A$ for unbounded perturbations $(\fdot, \f)\f$, i.e.~when $\f\notin \cH$, but $\f$ belongs to some $\cH_k$. Such perturbations are called \emph{singular}, and the case $\f\in \cH_{-1}\setminus \cH$ is probably the simplest case of a singular perturbation. 

Perturbations with $\f\in \cH_{-1}\setminus \cH$ are called \emph{form bounded}, the term form bounded used because the quadratic form of the perturbation $(\fdot, \f)\f$  is bounded by the quadratic form of the operator $I+|A|$. 

When $\f\notin \cH$, but $\f$ belongs to some $\cH_k$, we can define the quadratic form of the perturbed operator $A_\alpha= A +\alpha(\fdot, \f)\f$ on some dense subset of $\cH$. The question is whether or not this form gives rise to  a unique self-adjoint extension.

It is well-known that the answer is affirmative for form bounded perturbations. 


Without going into details about how the form bounded perturbation is defined, let us mention the main facts we will be using. The first one is the following \emph{resolvent formula}
\begin{eqnarray}\label{singres}
&&(A_\alpha-\lambda\OID)^{-1}f
=(A-\lambda\OID)^{-1}f-\frac{\alpha\left((A-\lambda\OID)^{-1}f,\f\right)}{1+\alpha\left((A-\lambda\OID)^{-1}\f,\f\right)}(A-\lambda\OID)^{-1}\f
\end{eqnarray}
which initially holds for $f\in \cH$, $\lambda\in \C\setminus\R$ (see, e.g.~equation (17) of \cite{AK1} or  Proposition 2.1 and Theorem 3.3 of \cite{KL}).

Note, the inner product $((A-\lambda \OID)^{-1} \f, \f)$ is well-defined for $\f\in \cH_{-1}(A)$ and $(A-\lambda \OID)^{-1} $ is an isomorphism between  $\cH_{r-2}(A)$ and $\cH_r(A)$. 
Probably the easiest way to see that is to invoke the Spectral Theorem. 


The following three well-known lemmata are corollaries of the resolvent formula \ref{singres}. 

\begin{lem}\label{l-halpha}
The resolvent formula \eqref{singres} can be extended to $f\in \cH_{-1}(A)$.

Moreover, for any $\lambda\in \C\setminus\R$ the operator $(A_\alpha-\lambda \OID)^{-1}$ is an isomorphism between  $\cH_{-1}(A)$ and $\cH_1(A)$.
\end{lem}
\begin{proof}
Take $f\in \cH_{-1}(A)$. We have $(A -\lambda \OID)^{-1} f \in \cH_1(A)$. So the right hand side of \eqref{singres} defines a bounded operator from $\cH_{-1}(A)$ to $\cH_1(A)$ (for $\lambda\in\C\setminus\R$). 

To complete the proof, take a sequence of vectors $f_n\in \cH$, $n\in \N$ such that $f_n\to f$ in the norm of $\cH_{-1}(A)$. The boundedness of the right side of \eqref{singres} implies that the sequence $g_n = (A_\alpha-\lambda \OID)^{-1} f_n$ converges in $\cH_1(A)$. Let $g$ be its limit. 

The boundedness of the right hand side of the identity \eqref{singres} implies the estimate \linebreak$\|g\|\ci{\cH_{1}(A)}\le C \|f\|\ci{\cH_{-1}(A)}$.  
Since $A_\alpha -\lambda \OID \in B(\cH_1(A), \cH_{-1}(A))$, we can conclude that $(A_\alpha -\lambda \OID) g = f$. 

The second statement follows trivially from the first one. 
\end{proof}

Let us recall that a vector $\f$ is called \emph{cyclic} for a self-adjoint operator $A$, if the span of the vectors $(A-\lambda \OID)^{-1} \f$, $\lambda \in \C\setminus\R$ is dense in $\cH$. Note that for this definition, one does not need to assume that $\f\in\cH$, but only that $(A-\lambda \OID)^{-1} \f\in\cH$, i.e.~that $\f\in \cH_{-2}(A)$. 
If $\f\in \cH_{-1}(A)$, then for $A_\alpha =A+ \alpha(\fdot, \f)\f$ we trivially have $(A_\alpha-\lambda \OID)^{-1} \f\in \cH$. So $\f\in \cH_{-2}(A_\alpha)$. 
\begin{lem}\label{cyc}
Let $\f\in \cH_{-1}(A)$ be a cyclic vector for $A$ and let $A_\alpha= A+ \alpha(\fdot, \f)\f$. Then 
$\f$ is cyclic for $A_\alpha$.
\end{lem}
\begin{proof}
Recall Lemma \ref{l-halpha}. So since
$$
\frac{\alpha\left((A-\lambda\OID)^{-1}\f,\f\right)}{1+\alpha\left((A-\lambda\OID)^{-1}\f,\f\right)} \ne 1 \qquad \forall \lambda \in \C\setminus \R, 
$$
the resolvent formula \eqref{singres} implies that $(A_\alpha -\lambda \OID)^{-1} \f = c(\lambda) (A-\lambda \OID)^{-1}\f$, $c(\lambda) \ne 0$ for all $\lambda \in \C\setminus \R$.
\end{proof}

\begin{lem}\label{l-falpha}
If $\f\in \cH_{-1}(A)$, then   $\f\in\cH_{-1}(A_\alpha)$ for all $\alpha\in \R$. In particular, we have $\int_\R \frac{\dd\mu_\alpha (s)}{1+|s|}<\infty$.
\end{lem}

\begin{rem*} If the operator $A$ is \emph{semibounded}, i.e.~if $A \ge a \OID$ for some $a\in \R$, then the proof of the lemma is almost trivial. Indeed, if $A$ is semibounded, then $A_\alpha$ is semibounded, and for semibounded operators $f\in \cH_{-1}(A)$ if and only if $( (A-\la \OID)^{-1} f,f)$ is  defined and bounded for some (or equivalently, for all) $\la\in \C\setminus \R$. 
\end{rem*}

We learned the proof  below, which works for the general case, from  Pavel Kurasov.

\begin{proof}[Proof of Lemma \ref{l-falpha}]
Recall that we assume $\f\in\cH_{-1}(A)$. Define
\[
F_\alpha(z):=((A_\alpha- z \OID)^{-1}\f,\f)=\int_\R\frac{1}{x-z}\,\dd\mu_\alpha(x)
\]for all $z\in \C\setminus \R, \alpha\in\R$. It is not hard to see that
\begin{eqnarray}\label{falpha}
\quad\quad\forall K>0 \,\,\,\exists C(K)>0\,\,\,:\,\,\frac{C(K)^{-1}}{1+|x|}\le\int_K^\infty \text{Im}\frac{1}{x-iy}\,\frac{\dd y}{y}\le \frac{C(K)}{1+|x|}\,.
\end{eqnarray}

Further we have the statement
\begin{eqnarray}\label{ImF}
\exists C\quad\forall |y|\ge C\quad:\quad|\im F(iy)|\sim|\im F_\alpha(iy)|\,,
\end{eqnarray}
where $F:=F_0$.
For the proof of statement \eqref{ImF}, first notice that for $|y|\ge 1$ it holds $\left|\frac{1}{iy-x}\right|\le\left|\frac{1}{i-x}\right|$ by a geometric argument. Since $\f\in\cH_{-1}(A)$, $\int\frac{\dd\mu(x)}{|i-x|}<\infty$. By the Dominated Convergence Theorem, we obtain
\[
\lim_{y\to\infty} |F(iy)|
\le
\lim_{y\to\infty} \int\left|\frac{1}{iy-x} \right|\dd\mu(x)
=
 \int\lim_{y\to\infty}\left|\frac{1}{iy-x} \right|\dd\mu(x)
=0.
\]
Recall the Aronszajn--Krein formula $F_\alpha = \frac{F}{1+\alpha F}$ which follows from the resolvent formula \eqref{singres}, see e.g.~equation (15) of \cite{AK1}. To see statement \eqref{ImF} note that
\[
\im F_\alpha = \im \frac{F}{1+\alpha F}
=\frac {\im F}{|1+\alpha F|^2}.
\]

Let us complete the proof that $\f\in\cH_{-1}(A_\alpha)$. The inclusion $\f\in\cH_{-1}(A)$ means that $\int_\R \frac{\dd\mu(x)}{1+|x|}<\infty$. By the right inequality of \eqref{falpha} for $\alpha=0$ it follows that $\int_K^\infty \frac{\text{Im} F(iy)}{y}\,\dd y<\infty$ for all $K>0$. For the interchange of order of integration, note that the latter integrand is positive for all $y$. According to \eqref{ImF} it follows that $\int_K^\infty \frac{\text{Im} F_\alpha(iy)}{y}\,\dd y<\infty$ for all $K>C$. By the left inequality of \eqref{falpha}, we obtain $\int_\R \frac{\dd\mu_\alpha(x)}{1+|x|}<\infty$, that is $\f\in\cH_{-1}(A_\alpha)$.
\end{proof}

\subsection{Proof of Theorem \ref{t-repr-V} for unbounded $A$}\label{iddd}
Recall that $A=M_t$ in $L^2(\mu)$, $A_\alpha= A +\alpha (\fdot, \f)\f$, $\f\equiv\ID_t$ and that $V_\alpha  A_\alpha=M_s V_\alpha  $, where $M_s$ is the multiplication by the independent variable $s$ in $L^2(\mu_\alpha )$. Recall also that $V_\alpha \ID_t=\ID_s$. 
Using the resolvent equality \eqref{singres} for $f=\f=\ID_t$ we get
\begin{align*}
&(M_s-\lambda\OID)^{-1}\ID_s
 =V_\alpha (A_\alpha-\lambda\OID)^{-1}V_\alpha ^{-1}\ID_s =V_\alpha (A_\alpha-\lambda\OID)^{-1}\ID_t
\\
=&
\left[1+\alpha \left((M_t-\lambda\OID)^{-1}\ID_t,\ID_t\right)\ci{L^2(\mu)}\right]^{-1}V_\alpha (M_t-\lambda\OID)^{-1}\ID_t
\end{align*}
for $\lambda\in \C\setminus\R$. So multiplying both sides by the term in square brackets and recalling that $(M_x-\lambda \OID)^{-1} \ID_x = (x-\lambda)^{-1}$, we have
\begin{eqnarray*}
&&V_\alpha\, \frac{1}{t-\lambda}=\left[1+\alpha\int\ci\R\frac{\dd \mu(t)}{t-\lambda}\right]\frac{1}{s-\lambda}\qquad\forall\lambda\in\C\setminus\R.
\end{eqnarray*}

Rewriting $\frac{1}{s-\lambda}\cdot\frac{1}{t-\lambda}=-\left(\frac{1}{s-\lambda}-\frac{1}{t-\lambda}\right)\frac{1}{s-t}$ we obtain
\begin{align*}
V_\alpha\, \frac{1}{t-\lambda}=\frac{1}{s-\lambda}-\alpha\int\ci\R\frac{\frac{1}{s-\lambda}-\frac{1}{t-\lambda}}{s-t}\,\dd \mu(t)
\end{align*}
for $\lambda\in\C\setminus\R$. By linearity we get that  formula \eqref{repr-V} holds for $f$ in the space
\[
B:=\spa\left\{\frac{1}{t-\lambda_k}\,:\,\lambda_k\in \C\setminus\R\right\}.
\]

Let us show that formula \eqref{repr-V} holds on $C_0^1(\R)$. 
Let $f\in C_0^1(\R)$, $\supp f\subset [-L,L]$, and let 
$P_\e$ be the Poisson kernel, $P_\varepsilon(x)= \frac{1}{\pi}\frac{\varepsilon}{x^2+\varepsilon^2}$.

Assume for a moment that the formula \eqref{repr-V} holds for the functions of form $P_\e*f$, $f\in C^1_0$. 
Convolution $P_\varepsilon* f$ converges to $f$ uniformly on $\R$. So $|P_\e*f(x)|\le C$ ($C$ does not depend on $\e$ and $x$) for all sufficiently small $\e$. Moreover, for  $|x|>2L$ we have $|(P_\e*f) (x)| \le C\e/x^2$, so  $|(P_\e*f) (x)| \le C/(1+|x|)$.

Since $(P_\e*f)'= P_\e * f'$, we conclude $(P_\e*f)' \to f'$ uniformly in $\R$, so $|(P_\e*f)'(x)|\le C$ for all sufficiently small $\e$. If $\e_n\searrow 0$, then the functions $f_n=P_{\e_n}* f$ satisfy the assumptions of Lemma \ref{lim-repr-V}, and \eqref{t-repr-V} holds for $f$.

To complete the proof of Theorem \ref{t-repr-V}, we need to show that formula \eqref{repr-V} holds for the functions of the form $P_\e*f$, $f\in C^1_0$.

Let us (for a fixed $\e>0$) approximate the convolution $g(x)=P_\e*f(x) =\int P_{\e}(x-t) f(t)\,dt$ by its Riemann sums. 

Since $P_\varepsilon(x-t)=\frac{1}{2\pi i}\left[\frac{1}{t - i\varepsilon-x}-\frac{1}{t+i\varepsilon-x}\right]$, the Riemann sums $g_n(x)$ can be chosen to be elements of $B$. So formula \eqref{repr-V} holds for $g_n$. Uniform continuity and boundedness of $f$ and $P_\e$ imply that $g_n\rightrightarrows g$. It is also easy to see that for $|x|>2L$ we can estimate $|g_n(x)|\le C/x^2$, thus $|g_n(x)|\le C/(1+|x|)$. 

Finally, taking the derivative we get the uniform estimate $|g_n'|\le C$. Notice, that $C=C(\e)$ here, we do not need uniform in $\e$ estimate.

Functions $g_n$ satisfy the assumptions of  Lemma \ref{lim-repr-V} and we can extend formula \eqref{repr-V} to functions of form $P_\e*f$, $f\in C^1_0$.   \hfill\qed

\subsection{Proof of the Rigidity Theorem \ref{rigTM}}\label{PROOFR}
Assume the hypotheses of the rigidity theorem, Theorem \ref{rigTM}, are satisfied.

Recall that $M_t$ and $M_s$ denote the multiplication operators by the independent variable in $L^2(\mu)$ and $L^2(\nu)$, respectively. 
 Note that if $M_s$ is unbounded, \emph{commuting} with $M_s$ means commuting with its spectral measures, or equivalently, with its resolvent. 

We utilize two lemmata.

\begin{lem}
\label{l-commute}
With the assumptions of Theorem \ref{rigTM} 
operator $VV^*$ commutes with $M_s$. In particular, we have $VV^*=M_{\psi}$ for some $\psi\in L^\infty(\nu)$.
\end{lem}

\begin{proof}
Let us first present an easier proof for the  case of bounded and compactly supported measures $\mu$ and $\nu$.

Let us begin by showing that 
\begin{eqnarray}\label{MS}
M_sV&=&V[M_t+\alpha(\fdot,\ID_t)\ci{L^2(\mu)} \ID_t].
\end{eqnarray}

Notice, that we can extend formula \eqref{repr-V} from $C^1_0$ to polynomials by multiplying the polynomials by an appropriate cut-off function $h\in C^1_0$, $h\equiv 1$ on $\supp \mu\cup\supp\nu$.

Let us prove \eqref{MS} for monomials $t^n$. 
For $f\equiv\ID_t$, formula \eqref{repr-V} yields $V\ID_t=\ID_s$.  Then application of \eqref{repr-V} to $ t^n$ and $t^{n+1}$ we get for $n\ge1$
\begin{align*}
&(M_sVt^n)(s)-(VM_tt^{n})(s)= s (Vt^n)(s)-(Vt^{n+1})(s)\\
=&-\alpha\int \left[\frac{s(s^n-t^n)}{s-t} - \frac{s^{n+1}-t^{n+1}}{s-t}\right]\, d\mu(t)\\
=&\alpha\int t^n\,\dd\mu(t)=\alpha (t^n,\ID_t)\ci{L^2(\mu)}\ID_s.
\end{align*}
So \eqref{MS} holds for  monomials $t^n$.

By linearity and continuity \eqref{MS} holds on polynomials. By assumption the polynomials are dense in $L^2(\mu)$, operator $V:L^2(\mu)\to L^2(\nu)$ is bounded and measures $\nu$, $\mu$ are bounded and of compact suppport. Therefore \eqref{MS} holds as an operator from $L^2(\mu)$ to $L^2(\nu)$.

Denoting $A_\alpha= M_t +\alpha(\fdot, \ID_t)\ID_t$ we rewrite \eqref{MS} as 
$M_sV=VA_\alpha$, and take the adjoint we get $V^*M_s=A_\alpha V^*$. So $VV^*$ commutes with $M_s$:
$$
M_sVV^*=VA_\alpha V^* = VV^*M_s.
$$

To prove the theorem in the general case we need an analogue of \eqref{MS} with resolvents instead of the operators $M_s$ and $A_\alpha$, see \eqref{OPEQ}.

First, taking a test function $f\in C^1_0$, $f\ge 0$, $\|f\|\ci{L^2(\mu)}>0$, and noticing that $|(Vf)(s)|\ge C/|s|$  for large $|s|$, we can see that the boundedness of the operator $V$ implies that $\int(1+s^2)^{-1}\,d\nu(s)<\infty$.

Next, we want to show that the representation formula \eqref{repr-V} holds on functions of the form $(t-\lambda)^{-1}$ for all $\lambda\in \C\setminus\R$.

Take $f(t)=(t-\lambda)^{-1}$. Notice that $f \in L^2(\mu)\cap L^2(\nu)$. Consider a family of cut-off functions $h_n$, $n\in\N $, such that $0\le h_n \le 1$, $h_n \equiv 1 $ on $[-n, n]$ and $|h'_n(t)|\le 1$.  Then for each $\lambda \in \C\setminus \R$ the family of functions $\{f_n\}$,  $f_n(t): = h_n(t)(t-\lambda)^{-1}$ satisfies the assumptions of Lemma \ref{lim-repr-V}, so the representation formula  \eqref{repr-V} holds for the functions $f$, $f(t) =(t-\lambda)^{-1}$. 

With this extension of formula \eqref{repr-V} we prove an identity that is an analog to \eqref{MS}. Namely, for $\lambda\in\C\setminus\R$ we have
\begin{eqnarray}\label{OPEQ}
V(A_\alpha-\lambda \OID)^{-1}&=&(M_s-\lambda)^{-1}V
\end{eqnarray}
on $L^2(\mu)$, where $A_\alpha=M_t+\alpha(\fdot,\ID_t)\ID_t$.

To show this, fix $\lambda\in\C\setminus\R$. Since $\frac{1}{s-\lambda}-\frac{1}{t-\lambda}=\frac{s-t}{(t-\lambda)(s-\lambda)}$, the representation formula \eqref{repr-V} gives us 
\[
(V (t-\lambda)^{-1})(s)
=\left[1+\alpha((t-\lambda)^{-1},\ID_t)\right](s-\lambda)^{-1}.
\]
That is
\[
V (M_t-\lambda\OID)^{-1} \ID_t
=\left[1+\alpha((M_t-\lambda\OID)^{-1}\ID_t,\ID_t)\right](M_s-\lambda\OID)^{-1} \ID_s.
\]

Due to resolvent equality \eqref{singres}, we get
\begin{align*}
&(M_s-\lambda\OID)^{-1} \ID_s
=
V \left[1+\alpha((M_t-\lambda\OID)^{-1}\ID_t,\ID_t)\right]^{-1} (M_t-\lambda\OID)^{-1} \ID_t\\
=&
V \left[1-\frac{\alpha((M_t-\lambda\OID)^{-1}\ID_t,\ID_t)}{1+\alpha((M_t-\lambda\OID)^{-1}\ID_t,\ID_t)}\right] (M_t-\lambda\OID)^{-1} \ID_t\\
=&V(A_\alpha-\lambda \OID)^{-1}\ID_t.
\end{align*}

For $\tau\in\C\setminus\R$, we know the (usual) resolvent identity
\begin{eqnarray*}
(A_\alpha-\lambda \OID)^{-1}(A_\alpha-\tau \OID)^{-1}
&=&
\left[(A_\alpha-\lambda \OID)^{-1}-(A_\alpha-\tau \OID)^{-1}\right](\lambda-\tau)^{-1}.
\end{eqnarray*}
Combination of the latter two equations yields
\[
V(A_\alpha-\lambda \OID)^{-1}(A_\alpha-\tau \OID)^{-1}\ID_t=\frac{1}{(s-\lambda)}V(A_\alpha-\tau \OID)^{-1}\ID_t.
\]
Identity \eqref{OPEQ} now follows from cyclicity of $\ID_t$ for $A_\alpha$, see Lemma \ref{cyc}.

Writing identity \eqref{OPEQ} for $\bar\lambda$ instead of $\lambda$ and taking the adjoint ,we have
\begin{eqnarray}
(A_\alpha-\lambda \OID)^{-1}V^*&=&V^*(s-\lambda)^{-1}.\label{commu4}
\end{eqnarray}

Combination of  \eqref{OPEQ} and \eqref{commu4} yields
\[
VV^*(M_s-\lambda\OID)^{-1}=(M_s-\lambda\OID)^{-1}VV^*
\qquad \forall\lambda\in\C\setminus\R,
\]
i.e.~$VV^*$ commutes with the spectral measures of $M_s$.

The second statement is a standard result in operator theory.
\end{proof}

\begin{lem}
\label{l-kernel}
Under the assumptions of Theorem \ref{rigTM},  $\Ker V^*=\{0\}$. 
\end{lem}
\begin{proof}
Since $\Ker V^* =\Ker VV^*$ and $VV^*$ commutes with $M_s$ (so $VV^*$ is a multiplication operator $M_\psi$), the kernel $\Ker V^*$ is a spectral subspace of $M_s$. Namely, there exists a Borel subset $E\subset \R$ such that
\[
\Ker V^* = \{f \in L^2(\nu)\,:\, \chi\ci{\R\backslash E} f =0\}.
\]

Assume $\Ker V^*\neq \{0\}$. Then $\nu(E)>0$. We obtain a contradiction by constructing a function $f\in\{f \in L^2(\nu)\,:\, \chi\ci{\R\backslash E} f =0\}$ such that $f\notin\Ker V^*$.

By assumption $\supp\mu$ consists of at least two points. Let $a\in\R$ such that there exist $I_1\Subset (-\infty,a), I_2\Subset  (a, \infty)$ with $\mu(I_1)>0, \mu(I_2)>0$. We need to consider two cases.

If $\nu(E\cap[a, \infty))>0$, we can pick $b\in\R$ such that $\nu(E\cap[a,b])>0$. Let $f=\chi\ci{E\cap[a,b]}$. Recall that $\int (1+s^2)^{-1}\,\dd\nu(s)<\infty$ (see proof of Lemma \ref{l-commute}). Hence $f\in L^2(\nu)$ and $\chi\ci{\R\backslash E} f =0$. Take $g\in C_0^1$ such that $g|\ci{I_1}=1$ and so that $g$ and $f$ have separated compact support. We have
\[
(f,Vg)\ci{L^2(\nu)}=\int\ci{E\cap[a,b]}\int\ci{I_1}\frac{f(s)\overline{g(t)}}{s-t}\,\dd\mu(t)\dd \nu(s)>0,
\]
since $\int\ci{I_1}\frac{\overline{g(t)}}{s-t}\,\dd\mu(t)>0$ for all $s\in E\cap[a,b]$.

Because $(V^*f,g)\ci{L^2(\mu)}=(f,Vg)\ci{L^2(\nu)}$, we have $f\notin\Ker V^*$.

Consider the case $\nu(E\cap[a, \infty))=0$. Recall $\nu(E)>0$. So $\nu(E\cap(\infty, a])>0$ and an analogous argument 
yields the desired contradiction.

The assumption $\Ker V^*\neq \{0\}$ was wrong. 
\end{proof}

With these two lemmata, we prove the rigidity theorem, Theorem \ref{rigTM}.

\begin{proof}[Proof of Theorem \ref{rigTM}]
Assume the hypotheses of Theorem \ref{rigTM}. In particular, $\Ker V=\{0\}$. With Lemmata \ref{l-commute} and \ref{l-kernel} we have $VV^*=M_{\psi}$, $\psi\in L^\infty(\nu)$, and it holds $\Ker V^*=\{0\}$.

Let us conclude the first statement.
Since $VV^*\ge0$, we have $\psi\ge0$ and the existence of operator $|V^*|=(VV^*)^{1/2} =M_{\psi^{1/2}} $. Writing polar decomposition, we get $V^* = \widetilde U |V^*|$ for some partial isometry $\widetilde U$. Note that $h^{-1}=\psi^{1/2}\in L^\infty(\nu)$. Taking $U:=\widetilde U^* = M_{h}V$. It remains to show that $\widetilde U$ is a unitary operator. We have $\Ker \widetilde U=\Ker V^*=\{0\}$. Let us show surjectivity. From $\Ker V=\{0\}$ and $h^{-1}\in L^\infty(\nu)$ it follows $\Ker \widetilde U^*=\{0\}$. By definition (polar decomposition) we have that $\Ran \widetilde U$ is closed. So also $\Ran \widetilde U=[\Ker \widetilde U^*]^\perp=L^2(\mu)$ and $\widetilde U$ is unitary.

Let us show the second part of the rigidity theorem, namely $UA_\alpha = M_s U$, where $M_s$ is the multiplication by the independent variable $s$ in $L^2(\nu)$. Consider the case of bounded $A$. From the proof of the first statement we extract $U= M_{\psi^{-1/2}}V$ and $\psi^{1/2}\in L^\infty(\nu).$ Substitution of $V$ into identity \eqref{MS} yields $M_sM_{\psi^{1/2}}U=M_{\psi^{1/2}}UA_\alpha$. Because multiplication operators commute, we get the second part of the rigidity theorem for bounded operators.
The unbounded case follows in analogy using \eqref{OPEQ} instead of \eqref{MS}.
\end{proof}

\section{Singular integral operators}\label{SIO}
Functions $f$ and $g$ are said to be of \emph{separated compact supports}, if $\supp f$ and $\supp g$ are compact sets and $\dist(\supp f,\supp g)>0$.

Let $K(s, t)$ be a function (kernel) which is bounded on each set $\{(s, t): |s-t|>\e\}$, $\e>0$. 

By a \emph{singular integral operator} (see \cite{NTV}), henceforth referred to as SIO, $T: L^2(\mu)\to L^2(\nu)$ with \emph{kernel} $K(s,t)$ we mean a bounded operator $T: L^2(\mu)\to L^2(\nu)$ such that for $f\in L^2(\mu)$ and $g\in  L^2(\nu)$ with separated compact supports
\[
(Tf,g)\ci{L^2(\nu)}=\iint K(s,t)f(t)\overline{g(s)}\,\dd\mu(t)\,\dd\nu(s).
\]
Notice, due to the condition of separated compact supports, the integral 
is well-defined.

\subsection{Unitary operator $V_\alpha$ is a singular integral operator}

\begin{lem}\label{VSIO}
Operator $V_\alpha : L^2(\mu)\to L^2(\mu_\alpha )$ from Theorem \ref{t-repr-V} is a SIO with kernel $K(s,t)=-\alpha(s-t)^{-1}$, i.e.~
\begin{equation}
\label{SIO-1}
(V_\alpha f,g)\ci{L^2(\mu_\alpha )}=-\alpha\iint \frac{f(t)\overline{g(s)}}{s-t}\,\dd\mu(t)\,\dd\mu_\alpha (s)
\end{equation}
for all $f\in L^2(\mu)$ and $g\in L^2(\mu_\alpha )$ with separated compact supports. 
\end{lem}

\begin{proof}

Formula \eqref{repr-V} implies that \eqref{SIO-1} holds
for  $f\in C^1_0$  and $g\in L^2(\mu_\alpha )$ if $f$ and $g$ have separated compact supports. 

To show that the same formula holds for arbitrary $f\in L^2(\mu)$ and $g\in L^2(\mu_\alpha )$ with separated compact supports, let us take a compact set $K$ such  that $\supp f\Subset K$ and $\dist(K,\supp g)>0$ and a sequence  $\{f_n\}$ of   $C^1_0$ functions so that $\supp f_n\subset K$ for all $n$ and such that $f_n\to f$ in $L^2(\mu)$.

Trivially $\lim_{n\to\infty}(V_\alpha f_n, g) = (V_\alpha f,g)$. Since $|s-t|^{-1}\le 1/\dist\{K, \supp g\}$ for $t\in K$, $s\in \supp f$, one can easily see that 
$$
\lim_{n\to \infty} \int \frac{f_n(t)\overline{g(s)}}{s-t}\,\dd\mu(t)\,\dd\mu_\alpha (s) = \int \frac{f(t)\overline{g(s)}}{s-t}\,\dd\mu(t)\,\dd\mu_\alpha (s),
$$
which proves the lemma. \end{proof}

\subsection{Cauchy transform  acting $L^2(\mu)\to L^2(\mu_\alpha)$ and its regularizations}

It is well-known in the theory of singular integral operators, that if a singular operator $T$ with  a Calderon--Zygmund kernel%
\footnote{Calderon--Zygmund means that the kernel $K$ satisfies some growth and smoothness estimates. Without giving the definition let us only mention that $1/(s-t)$ is one of the classical examples of a Calderon--Zygmund kernel.}
 $K$ is bounded on $L^2$, then the truncated operators $\widetilde T_\e$, where
$$
\widetilde T_\e f (s) = \int_{|t-s|>\e} K(s, t) f(t) \,dt,
$$
are uniformly (in $\e$) bounded. Also, this fact remains true, if instead of truncations, one considers any reasonable regularization of the kernel $K$. 

However, the classical theory does not apply in our case, because we integrate with respect to the measure $\mu$ which does not satisfy the doubling condition. Moreover, even the recently developed theory, see \cite{NTV}, of singular integral operators on non-homogeneous spaces (i.e.~with non-doubling measure) does not work here, because, first this theory works only for one weighted case (the same measure in the target space), and second, the measure $\mu$ has to satisfy a growth condition ($\mu([a-\e, a+\e])\le C\e$ uniformly in $a$ and $\e$).

And the measure $\mu$ appearing in our situation can be any Radon measure. So no known result about singular integrals can be applied here.

Nevertheless, it still can be shown that the following regularized operators are uniformly bounded operators acting from $L^2(\mu)$ to $L^2(\mu_\alpha)$.

Let $T_\e=(T_\mu)_\e$, $\e>0$ be the integral operator with kernel $(s-t+i\e)^{-1}$, 
\begin{align}\label{def-Te}
T_\e f(s) := \int \frac{f(t)}{s-t+i\e} \,d\mu(t),
\end{align}
and let $\widetilde T_\e=(\widetilde T_\mu)_\e$ be the truncated operator, 
$$
\widetilde T_\e f(s) := \int_{|t-s|>\e} \frac{f(t)}{s-t} \,d\mu(t) . 
$$
Note, it is trivial that both $T_\e$ and $\widetilde T_\e$  are well-defined for compactly supported $f$. It is also not hard to show - using Cauchy--Schwartz - that, if $\int(1+x^2)^{-1} \,d\mu(x)<\infty$, then the operators are well-defined for all $f\in L^2(\mu)$.

\begin{theo}\label{reg-2}
Let $\mu$ and $\mu_\alpha$ be the spectral measures of $A$ and $A_\alpha$, correspondingly.

Then the regularized operators $T_\e= (T_\mu)_\e :L^2(\mu)\to L^2(\mu_\alpha )$ defined by  \eqref{def-Te} are uniformly bounded  $\| T_\e\|\ci{L^2(\mu)\to L^2(\mu_\alpha)} \le 2|\alpha|^{-1}$.

Moreover, the weak limit $T$ of $T_\e$ exists as $\e\to 0^+$, and operator $V_\alpha$ has the alternative representation
\begin{align}\label{Tweak}
V_\alpha f(s)=f(s)(\ID-\alpha \,T\ID)+\alpha \,Tf
\end{align}
for all  $f\in L^2(\mu)$.

Finally, for any $f\in C^1_0$
\[
\lim_{\e \to 0^+} (T_\e f )(s) = T f(s)
\]
$\mu_\alpha$-a.e. 
\end{theo}

\begin{rem*}
If $\ID \notin L^2(\mu)$, the function $T\ID$ can be defined, for example, by duality,
\[
\int  T1 \overline f d\mu_\alpha = \int \overline{ T^*f} d\mu
\]
for all compactly supported $f\in L^2(\mu_\alpha)$. 
Note, since $\int(1+|x|)^{-1} d\mu(x)<\infty$, the integral $ \int \overline{ T^*f} d\mu$ is well-defined. 
It is easy to see from the proof that $T\ID$ coincides with $-F(x+i0^+)$, $F(x +i0^+) := \lim_{\e\to 0^+} F(x+i\e)$, where
\[
F(z) = \int \frac{ d\mu(t) }{t-z}\,.
\]
\end{rem*}
\begin{rem*}
For purely singular measure $\mu$ the $\mu_\alpha$-a.e.~convergence of $T_\e f$ for all $f\in L^2(\mu)$ (not only for $f\in C_0^1$)
 was settled by Poltoratski's Theorem in \cite{NONTAN}.  Apparently, as it came out of our communications with A.~Poltoratski and other experts in this area, it is possible to prove $\mu_\alpha$-a.e.~convergence for all $f\in L^2(\mu)$ in the general case, although it is hard to present a formal reference. 
 
However, for our purposes a simpler fact of $\mu_\alpha$-a.e.~convergence for all $f\in C^1_0$,  is sufficient. 
\end{rem*}

\begin{proof}[Proof of Theorem \ref{reg-2}]
To prove the first statement, let $V_\alpha :L^2(\mu)\to L^2(\mu_\alpha )$ be the spectral representation of $A_\alpha$ from Theorem \ref{t-repr-V}.
Using formula \eqref{repr-V} it is easy to see that for all real $a$ and compactly supported $f\in C^1$ it holds
\begin{align*}
V_\alpha f(s)-e^{ias}V_\alpha [e^{-iat}f] (s)
=
\alpha \int\frac{f(t)(1-e^{ia(s-t)})}{s-t}d\mu(t).
\end{align*}
Note that the kernel $\frac{(1-e^{ia(s-t)})}{s-t}$ is bounded, so the integral is well-defined for compactly supported functions $f\in L^2(\mu)$. 

Since $\|V_\alpha\|=1$ and  multiplication by $e^{-iax}$ is a unitary operator in $L^2(\mu)$ and $L^2(\mu_\alpha)$, we have
\[
\left\|\int\frac{f(t)(1-e^{ia(s-t)})}{s-t}d\mu(t)\right\|\ci{L^2(\mu_\alpha )}\le 2|\alpha|^{-1}\|f\|\ci{L^2(\mu)}.
\]

For $\e>0$ we have
$$
\e\int_{0}^\infty \frac{1-e^{ia(s-t)}}{s-t}\, e^{-\e a} da = \frac{1}{s-t+i\e} 
$$
and $\e\int_0^\infty e^{-\e a } \,da = 1$. So, by averaging the integral in the left side of  \eqref{SIO-2} over all $a\ge0$ with weight $\e e^{-\e a}$, we get
\[
\left\|\int\frac{f(t)}{s-t+i\e}d\mu(t)\right\|\ci{L^2(\mu_\alpha )}\le 2|\alpha|^{-1}\|f\|\ci{L^2(\mu)}
\]
for compactly supported $f\in C^1$ and all $\e>0$.

Let us show the existence of the weak limit of $T_\e$.

Take a    convergent  (in weak operator topology) sequence    $T_{\e_k} \to \widehat T $ ,  $\e_k\to 0$, as $k\to\infty$. 
For $f\in C_0^1$, we have that $T_\e f \to  T f$ pointwise $\mu_\alpha$-a.e.~for some operator $T$. Indeed, 
\[
T_\e f(s)
=
\int\frac{f(t)}{s-t+i\e}\,d\mu(t)
=
\int\frac{f(s)-f(t)}{s-t+i\e}\,d\mu(t)
-f(s) T_\e \ID
\]
and note that the integrand on the right hand side remains bounded as $\e\to0$ for compactly supported $C^1$ functions $f$. For the second term on the right hand side, recall that we denote by $w$ the density function of operator $A$'s spectral measure. Aronszajn--Donoghue's theory on rank one perturbations says that $- F(\fdot +i\e ) = T_\e \ID\to -\pi w$ a.e.~with respect to the Lebesgue measure and $- F(\fdot +i\e ) =T_\e \ID \to \alpha^{-1}$a.e.~with respect to $(\mu_\alpha)\ti{s}$~as $\e\to 0$, see e.g.~\cite{SIMREV}. So for $f\in C_0^1$ we have that $T_\e f \to   T f$ pointwise $\mu_\alpha$ almost everywhere.

Lemma \ref{wkeqae} below shows that $\widehat T f=  T f$ for all $f\in C_0^1$, so $\text{w.o.t.-}\lim_{k\to\infty}T_{\e_k} = \widehat T= T$. 

Since the operators  $T_\e$ are uniformly bounded, any sequence $\e_k\to 0$ has a subsequence $\e_{k_n}$ such that $T_{\e_{k_n}}$ converges in weak operator topology. As we discussed above  this limit must be $T$. And that means $\text{w.o.t-} \lim_{\e\to0} T_{\e} = T $.

Let us prove representation formula \eqref{Tweak}. 
Take  $f\in C_0^1$. 

By the Dominated Convergence Theorem we have
\[
\int\frac{f(s) - f(t)}{s-t}d\mu(t)
\,=\,
\lim_{\e\to0^+}\left[\int\frac{f(s)}{s-t+i\e}d\mu(t) - \int \frac{f(t)}{s-t+i\e}d\mu(t)\right]
\]
for all real $s$.

From the first part of the theorem, it follows that the second integral converges weakly in $L^2(\mu_\alpha)$ and $\mu_\alpha$-a.e.~to $Tf$ as $\e\to 0$.

It is an easy exercise to show that the first integral converges weakly in $L^2(\mu_\alpha)$ to $f T\ID = -f F(\fdot +i0^+) $, and $\mu_\alpha$-a.e.~convergence was shown above in the proof. 

Representation \eqref{Tweak} now immediately follows from \eqref{repr-V}. 
\end{proof}

The lemma below is well-known. We present the proof only for the sake of completeness. 

\begin{lem}\label{wkeqae}
Let $\eta$ be a measure.
If a sequence of functions $f_n$ converges to $f$ weakly in $L^2(\eta)$ and to $g$ pointwise $\eta$-a.e., then we have $f=g$ in $L^2(\eta)$.
\end{lem}

\begin{proof}
Recall that a closed convex subset of a Banach space is weakly closed (it is a simple corollary of the Hahn--Banach theorem). So we have $f\in\wclos(\conv\{f_{n}, f_{{n+1}}, \hdots\}) = \clos(\conv\{f_{n}, f_{{n+1}}, \hdots\})$ for all $n\in\N$. 
Hence for every $n$ there exists a non-negative sequence $\{\alpha_k^n\}\ci{k\in\N}$  with $\sum_{k\ge n} \alpha_k^n=1$ and such that $g_n=\sum_{k\ge n}\alpha_k^n f_{n}$ converge to $f$ in $L^2(\eta)$. Therefore, one can find a subsequence $g_{n_k}$ such that $g_{n_k}\to f$ $\eta$-a.e.  

On the other hand $\lim g_{n_k}(x) = \lim g_n(x) = g(x)$, so $f=g$ $\eta$-a.e.
\end{proof}

\subsection{Regularization of the Cauchy transform in the general case}
The situation we considered in the previous section is very special, because measures $\mu$ and $\mu_\alpha$ are rigidly related to each other. Theorem below shows that for very general measures, a rather natural and weak assumption of boundedness implies the uniform boundedness of the regularized operators.

Let us recall that two Borel measures $\mu$ and $\nu$ are called \emph{mutually singular} (notation  $\mu\perp\nu$) if they are supported on disjoint sets, i.e.~if there exist Borel sets $E$ and $F$ such that $E\cap F =\varnothing$ and $\mu(E^c) = \nu(F^c)=0$. 

\begin{theo}
\label{reg-1}
Let $\mu$ and $\nu$ be Radon measures on $\R$ such that for their singular parts $\mu\ti{s}\perp\nu\ti{s}$, and such that
\begin{equation}
\label{3.1} 
\left| \iint \frac{f(t)\overline{g(s)}}{s-t}\,\dd\mu(t)\,\dd\nu(s) \right| \le C\|f\|\ci{L^2(\mu)} \|g\|\ci{L^2(\nu)} 
\end{equation}
for all $f$ and $g$ with separated compact supports. 

Then for all $\e>0$ 
$$
\|T_\e f \|\ci{L^2(\nu)} \le 4 C \|f\|\ci{L^2(\mu)} \qquad \forall f\in L^2(\mu), 
$$
and the truncated operators $\widetilde T_\e:L^2(\mu)\to L^2(\nu)$ are also uniformly bounded. 
\end{theo}

\begin{rem*}
By a well-known Aronszajn--Donoghue theorem the singular parts of $\mu_\alpha$ and $\mu_\beta$ are mutually singular for all $\alpha, \beta \in\R$ with $\alpha \neq \beta$, see e.g.~\cite{SIMREV}, so the above theorem can be used in the situation when $\mu$ is the spectral measure of $A$ and $\nu=\mu_\alpha$ is the spectral measure of $A_\alpha$. 

On the other hand, it is not hard to show that the uniform boundedness of $T_\e$ implies that $\mu\ti s\perp\nu\ti s$, so Theorem  \ref{reg-2} gives a different proof of this Aronszajn--Donoghue theorem. 
\end{rem*}

\begin{proof}[Proof of Theorem \ref{reg-1} for $T_\e$]
Estimate \eqref{3.1} holds, if we replace function $f$ by $e^{-iat}f(t)$ and $g$ by $e^{-ias} g(s)$, $a\in \R$. So for all $a\in \R$
\begin{equation}
\label{SIO-2}
\left| \iint f(t)\overline{g(s)}\,\, \frac{1-e^{ia(s-t)}}{s-t} \,d\mu(t)d\nu(s) \right| \le 2 C  
\|f\|\ci{L^2(\mu)} \|g\|\ci{L^2(\nu)}
\end{equation}
(again for $f$ and $g$ with separated compact supports). 

Since for $\e>0$
$$
\e\int_{0}^\infty \frac{1-e^{ia(s-t)}}{s-t}\, e^{-\e a} da = \frac{1}{s-t+i\e} 
$$
and $\e\int_0^\infty e^{-\e a } \,da = 1$, we get  by averaging \eqref{SIO-2} over all $a\ge0$ with weight $\e e^{-\e a}$ that 
\begin{equation}
\label{SIO-3}
\left| \iint \frac{f(t)\overline{g(s)}}{s-t+i\e} \, d\mu(t)d\nu(s) \right| \le 2 C\|f\|\ci{L^2(\mu)} \|g\|\ci{L^2(\nu)}
\end{equation}
independent of $\e$. The lemma below shows that the estimate holds for arbitrary compactly supported functions, not necessarily with separated supports, which proves the theorem for $T_\e$. 
\end{proof}
 
\begin{lem}
\label{l-mix1}
Let $\mu$ and $\nu$ be Radon measures such that $\mu\ti{s}\perp\nu\ti{s}$. Let $T:L^2(\mu) \to L^2(\nu)$ be compact. 

If $|(Tf, g)|\le C\|f\|\ci{L^2(\mu)}\|g\|\ci{L^2(\nu)}$ for all pairs $f\in L^2(\mu)$ and $g\in L^2(\nu)$ with separated compact supports, then $\|T\|\le 2C$. 
\end{lem}

If one restricts everything to an interval $(-R, R)$, the integral operator with kernel $1/(s-t+i\e)$ is clearly compact. So
 Lemma \ref{l-mix1} gives the estimate $|(T_\e f, g)|\le 4C \|f\|\ci{L^2(\mu)}\|g\|\ci{L^2(\nu)}$ for compactly supported $f$ and $g$, which is all we need to prove Theorem \ref{reg-1}. 

\begin{proof}[Proof of Lemma \ref{l-mix1}]
Consider first the case when $\mu$ and $\nu$ are absolutely continuous with respect to Lebesgue measure.  Pick small $\delta>0$, define functions  $h_1= \mathbf{1}_{[0, 1/2-\delta]}$, $h_2= \mathbf 1_{[1/2, 1-\delta]}$ on $[0,1)$ and extend them to $1$-periodic functions on the whole real line.

For $f\in L^2(\mu)$ and $g\in L^2(\nu)$, define 
functions $f_n$, $g_n$ by
$$
f_n(t): = f(t) h_1(nt), \qquad g_n(s): = g(s) h_2(ns). 
$$
For each $n$, the functions $f_n$, $g_n$ have separated support. 
We claim that 
\begin{equation}
\label{mix1}
f_n \to (1/2-\delta) f, \qquad g_n\to (1/2-\delta) g
\end{equation}
weakly in $L^2(\mu)$ and $L^2(\nu)$, respectively, and that 
\begin{equation}
\label{mix2}
\|f_n\|^2\ci{L^2(\mu)} \to (1/2-\delta) \|f\|^2\ci{L^2(\mu)}, \quad
\|g_n\|^2\ci{L^2(\nu)} \to (1/2-\delta) \|g\|^2\ci{L^2(\nu)}
\end{equation}
as $n\to \infty$. 

Both statements follow immediately from the fact that for arbitrary $\phi\in L^1$ (with respect to the Lebesgue measure) and for  $h=h_1$ or $h=h_2$
$$
\lim_{n\to\infty} \int_\R \phi(t) h(nt) dt = (1/2-\delta) \int_\R \phi(t) dt. 
$$
This fact is trivial for characteristic functions of intervals, extends by linearity for their finite linear combinations and from this dense set to all $L^1$ by $\e/3$ Theorem, since the functionals $\phi \mapsto \int_\R \phi(t)h(nt) dt$ are uniformly bounded. 

Since $T$ is compact, the weak convergence of $f_n$ and $g_n$ implies that $(Tf_n, g_n)\to (1/2-\delta)^{2}(Tf,g)$. Therefore
\begin{align*}
(1/2-\delta)^{2}|(Tf, g)|= \lim_{n\to\infty} |(Tf_n,g_n)|& \le C\lim_{n\to\infty}   \|f_n\|\ci{L^2(\mu)}\|g_n\|\ci{L^2(\nu)}  
\\ & = (1/2-\delta)C\|f\|\ci{L^2(\mu)}\|g\|\ci{L^2(\nu)}, 
\end{align*}
so $\|T\|\le (1/2-\delta)^{-1} C$. Since $\delta$ can be arbitrary small, the conclusion of the lemma follows. 

The reasoning in the above paragraph works for general measures. So to prove lemma for the general case it is sufficient for arbitrary $f\in L^2(\mu)$, $g\in L^2(\nu)$ to construct functions $f_n$, $g_n$ satisfying \eqref{mix1}, \eqref{mix2} and such that for each $n$ the supports of $f_n$ and $g_n$ are separated.  

Let $E$ and $F$ be disjoint Borel subsets of Lebesgue measure zero supporting singular parts of $\mu$ and $\nu$, respectively, meaning that $\mu\ti{s}(E^c)=0$, $\nu\ti{s}(F^c)=0$. Denote $G:= (E\cup F)^c$.

Since Radon measures on $\R$ are inner regular, there exist compact subsets $E_n\subset E$, $F_n\subset F$, $G_n\subset G$ such that $\mu(E_n)\to \mu(E)$, $\nu(F_n)\to \nu(F)$, $\mu(G_n) + \nu(G_n) \to \mu(G)+\nu(G)$ as $n\to \infty$. 

Let $f_a= f\chi\ci G$, $g_a=g\chi\ci G$ be ``absolutely continuous'' parts of $f$ and $g$, and $f\ti{s} = f\chi\ci E$, $g\ti{s}= g\chi\ci F$ be the ``singular'' parts of $f$ and $g$. Take $\delta>0$ and define
\begin{align*}
f_n(t) &:= f_a(t) h_1(nt) \chi\ci {G_n}(t) + (1/2-\delta) f\ti{s} (t)\chi\ci {E_n}(t), \\
g_n(t) &:= g_a(t) h_2(nt) \chi\ci {G_n}(t) + (1/2-\delta) g\ti{s}(t) \chi\ci {F_n} (t).
\end{align*}
Clearly, for each $n$ supports of $f_n$ and $g_n$ are separated.

Let us show that $f_n\to (1/2-\delta)f$ weakly in $L^2(\mu)$. Clearly, due to absolute continuity of integral $\|f\ti{s}\chi\ci {E_n}-f\ti{s}\|\ci{L^2(\mu)}\to 0$ as $n\to \infty$. 

Take arbitrary $k\in L^2(\mu)$. Then 
$$
\int_\R f_a(t) h_1(nt) \chi\ci {G_n}(t) \overline{k(t)} \,d\mu(t) \to (1/2-\delta) (f_a, k)\ci{L^2(\mu)}
$$
because, as it was discussed above $f_a(t) h_1(nt)$ converges weakly to $(1/2-\delta) f_a$, and trivially $k\chi\ci{G_n}$ converges strongly to $k\chi\ci G$. 

As for the norms, it is not hard to show that 
$$
\lim_{n\to\infty} \|f_n\|^2\ci{L^2(\mu)} = (1/2-\delta)\| f_a\|\ci{L^2(\mu)}^2 + (1/2-\delta)^2 \|f\ti{s}\|\ci{L^2(\mu)}^2 \le (1/2-\delta)\|f\|\ci{L^2(\mu)}^2.
$$

Similarly $g_n\to (1/2-\delta)g$ weakly in $L^2(\nu)$ and $\lim_n \|g_n\|^2\ci{L^2(\nu)} \le (1/2-\delta) \|g\|^2\ci{L^2(\nu)}$. 

And the same reasoning as for the absolutely continuous case completes the proof.
\end{proof}

In order to show Theorem \ref{reg-1} for $\widetilde T_\e$, we prove the necessity of an $A_2$-type condition for $T_\e$ to be uniformly bounded.

\begin{lem}\label{measures}
Let $\mu$ and $\nu$ be Radon measures such that the operators $T_\e:L^2(\mu)\to L^2(\nu)$ are uniformly bounded.  Then there exists a constant $C>0$ such that 
\[
\int_\R \frac{(\im a ) \,d\mu(t)}{|t-a|^2} \int_\R \frac{(\im a)\, d\nu(t)}{|t-a|^2} \le C
\]
for all $a$, $\im a>0$. In particular
$|I|^{-2}\mu(I)\nu(I)\le C'<\infty$ for all intervals $I\neq\varnothing$.
\end{lem}

\begin{proof}[Proof of Lemma \ref{measures}]
Let $b_a(x)=\frac{x-a}{x-\bar a}$, $a\in\C$. Consider auxiliary operators $R_\e:= T_\e-M_{\overline{b_a}}T_\e M_{b_a}:L^2(\mu)\to L^2(\nu)$, where $M_\f$ is the multiplication operator, $M_\f f = \f f$.  Since $M_b$ and $M_{\overline b}$ are isometries, the operators $R_\e$ are uniformly bounded with respect to $\e$ and $a$.

Since 
\[
\frac{1}{s-t+i\e}-\frac{(s-\bar a)(t-a)}{(s-a)(s-t+i\e)(t-\bar a)}=
\frac{2i\im a(s-t)}{(s-t+i\e)(s-a)(t-\bar a)}, 
\]
we have for compactly supported $f\in L^2(\mu)$
\[
R_\e f\,(s)
=
\int\frac{2i\im a(s-t)f(t)}{(s-t+i\e)(s-a)(t-\bar a)}\, d\mu(t). 
\]

It follows from the Dominated Convergence Theorem that for compactly supported $f\in L^2(\mu)$, $g\in L^2(\nu)$
\[
\lim_{\e\to 0^+} (R_\e f, g)\ci{L^2(\nu)} = \iint \frac{ 2i (\im a ) f(t) \overline{g(s)} }{(s-a)(t-\overline a)} \, d\mu(t) d\nu(s), 
\]
so the weak limit $R_0 := \text{w.o.t.-}\lim_{\e\to0^+} R_\e$. Its norm can be easily computed (for example, the operator norm of a rank one operator coincides with its Hilbert--Schmidt norm):
\begin{equation}
\label{Muck-1}
\int_\R \frac{2 \im a}{|t-a|^2 }d\mu(t) \int_\R \frac{2 \im a}{|s-a|^2 }d\nu(s) = \|R_0\|^2\le 4\limsup_{\e\to0^+} \|T_\e\|^2<\infty. 
\end{equation}
But that is exactly the conclusion of the theorem. 

To prove the statement about intervals, 
take a non-empty interval $I$. Set $\im a=|I|$ and $\re a=1/2(\sup I-\inf I)$. Integrating in \eqref{Muck-1} only over $I\times I$  and using that $1/|t-a| \ge 1/(2|I|)$ for $t\in I$ we get that $|I|^{-2}\mu(I)\nu(I)\le C'<\infty$. 
\end{proof}

\begin{proof}[Proof of Theorem \ref{reg-1} for $\widetilde T_\e$]
To prove Theorem \ref{reg-1} for the operators $\widetilde T_\e$ it is sufficient  to show that the difference operators $T_\e-\widetilde T_\e$ are uniformly bounded.


The difference operator is  defined for compactly supported $f\in L^2(\mu)$ by
\[
(T_\e - \widetilde T_\e)f(s)
=\int K_\e (s-t) f(t)d\mu(t)
\]
where $K_\e(x)=(x+i\e)^{-1}-x^{-1}\chi\ci{[-\e,\e]^c}$. Note that the kernel $K_\e$ satisfies $|K_\e(x)|\le \frac{\sqrt{2}\e}{x^2+\e^2}\,$, so it can be majorated by a convex combination of functions $|I|^{-1}\chi\ci I$, 
\[
|K_\e(x) | \le \sum_k c_k(\e)  |I_k|^{-1} \chi\ci{I_k}(x) =: M_\e (x),  \qquad c_k(\e)\ge 0, \ \sum_k c_k(\e) \le C <\infty. 
\]

Clearly
\[
\left| \iint K_\e(s-t) f(t) \overline{g(s)} d\mu(t)d\nu(s) \right| 
\le 
\iint M_\e (s-t) |f(t)| \cdot |g(s)| d\mu(t)d\nu(s). 
\]
So, to prove uniform boundedness of $T_\e-\widetilde T_\e$, it is sufficient to show that the operators 
 $T\ci{I}:L^2(\mu)\to L^2(\nu)$ given by
\[
T\ci{I}f (s)=|I|^{-1}\int \chi\ci{I}(s-t)f(t)d\mu(t)
\]
are uniformly bounded.

%

To prove this uniform estimate  let $\cup\ci{k\in\Z}J_k$ be a cover of $\R$ by non-intersecting half open intervals of length $|J_k|=|I|$. Let $\widetilde J_k:=J_{k-1}\cup J_k\cup J_{k+1}$.

For all $s\in J_k$, we have
\[
T\ci{I} f(s)
\le
3 |\widetilde J_k|^{-1} \int\ci{\widetilde J_k} f d\mu   
\le 
3 \Bigl( |\wt J_k|^{-1} \int\ci{\widetilde J_k} |f|^2 d\mu  \Bigr)^{1/2} \left( |\wt J_k|^{-1} \mu(\wt J_k) \right)^{1/2}\,.
\]
(The last inequality is just Cauchy--Schwartz.) So we obtain
\[
\int_{J_k} |T\ci I f |^2 d\nu \le 9 |\wt J_k|^{-2} \mu(\wt J_k) \nu(J_k) \int_{\wt J_k} |f|^2 d\mu. 
\]
Summing over all $k$ and taking into account that $|\wt J_k|^{-2} \mu(\wt J_k) \nu(J_k) \le |\wt J_k|^{-2} \mu(\wt J_k) \nu(\wt J_k)\le C'$ and that each $x\in \R$ is covered by $3$ intervals $\wt J_k$, we get
\[
\int_\R |T\ci I f|^2 d\nu \le 27 C' \int_\R |f|^2 d\mu.
\]
\end{proof}

\section{Absence of singular spectrum}\label{absSEC}
In this section we are going to investigate the absence of the singular spectrum of the perturbed operator $A_\alpha$. 


For a complex-valued Borel measure $\eta$ on $\R$ such that $\int \frac{|d \eta(t)|}{1+t^2}< \infty$, let 
\[
K\eta(s) :=\lim_{\e\to 0^+} \int \frac{d \eta(t)}{s-t+i\e}\,.
\]
It is a standard fact that this limit exists  almost everywhere with respect to Lebesgue measure.

We will need the result below about the boundary values of the Cauchy transform of a measure,  
cf \cite{GOLU}, where it was proved for the case of the unit circle.  The case of the real line can be treated absolutely the same way. 
\begin{theo}
Let $I\subset \R$ be a bounded open interval.
Then 
\[
t \chi\ci{(\{|K\eta|>t\}\cap I)} \, dx\,\,\, \to\,\,\, 2 \chi\ci{I} \,d |\eta\ti{s}| + \chi\ci{\partial I}  \,d |\eta\ti{s}|
\]
in the weak$\,^*$-sense as $t\to \infty$.
\end{theo}

The following corollary is an immediate consequence of this theorem.

\begin{cor}\label{APOL-t}
If $I\subset \R$ is a bounded closed interval such that $\eta\ti{s}|\ci{I}\ne 0$, then there exists a $C>0$ such that $|\{|K\eta|>t\}\cap I|\ge C/t$ for large $t$.
\end{cor}
Assume the setting of rank one perturbations, see e.g.~Section \ref{setup}. Let 
$$
F(z) = \int_\R \frac{d\mu(x)}{x-z}, \qquad F_\alpha(z) = \int_\R \frac{d\mu_\alpha (x)}{x-z}
$$
where $\im z >0$,  $\mu$ and $\mu_\alpha$ are the spectral measures of $A$ and $A_\alpha$, respectively. 

By the   well-known Aronszajn--Krein formula $F_\alpha=F/(1+\alpha F)$. It is also well-known  that $\im K\nu =\lim_{\e\to 0^+} \im F(x+i\e) = w(x)$ a.e.~with respect to Lebesgue measure ($w$ is the density of the absolutely continuous part of $\mu$) and similarly for $F_\alpha$. 

Since 
\[
\im F_\alpha = \im \frac{f}{1+\alpha F} = \frac{\im F}{|1+\alpha F|^2} 
\le \frac1{|\alpha|^2 } \frac1{\im F}, 
\]
we can conclude that the density $w_\alpha$ of the absolutely continuous part of $\mu_\alpha$ satisfies 
\[
w_\alpha \le \frac1{\pi^2|\alpha|^2 } \frac1{w}\,.
\]
Combining this with Corollary \ref{APOL-t} one immediately get the following proposition.

\begin{prop}
Let for a bounded closed interval $I$
\[
\left| \{ x\in I: 1/w(x) > t\}\right| = o(1/t), \qquad t\to +\infty. 
\]
Then for all $\alpha \in \R$, $\alpha \ne 0$ the measures $\mu_\alpha$ do not have singular part on $  I $. 
\end{prop}
The above reasoning is probably well-known to specialists. We have learned it from E.~Abakumov (personal communication). 

Using the fact about uniform (in $\e$ not in $\alpha$) boundedness of the operators $T_\e=(T_\mu)_\e: L^2 (\mu)\to L^2(\mu_\alpha)$ we can get a  stronger result in this direction. 

For a bounded  interval $I$ and a weight $w$, define the \emph{distribution function} $D_w =D_{w, I}(t):=|\{x\in I: w(x) <t\}|$ of $w|_I$. Consider its inverse function, the \emph{increasing rearrangement} $w^*=w^*\ci I$ of $w|\ci I$, i.e.~$w^*=(D_{w})^{-1}$.

\begin{lem}\label{eqinfty}
Let $\mu$ and $\nu$ be Radon measures on $\R$ such that the operators $T_\e=(T_\mu)_\e: L^2 (\mu)\to L^2(\nu)$, 
$$
T_\e f (s) = \int_\R \frac{f(t)}{s-t + i\e} d\mu(t),
$$
are uniformly (in $\e$) bounded, and let $d\mu= wdt + d\mu\ti s$ be the Lebesgue decomposition of the  measure $\mu$. Assume that for a bounded closed interval $I$ the increasing rearrangement $w^* = w^*_I$ of $w|\ci I$ satisfies
\begin{align}\label{aintegral}
\int_0^\e x^{-2} w^*(x) \, dx =\infty
\end{align}
for some (all) $\e>0$. 

Then the measure $\nu$ does not have the singular part on $I$.
\end{lem}

\begin{lem}
Condition \eqref{aintegral} can equivalently be expressed in terms of the distribution function $D_w = D_{w, I}$ as 
\begin{equation}
\label{aint-1}
\int_0^\delta \frac1{D_w(y)} dy = \infty. 
\end{equation}
\end{lem}
\begin{proof}
If $w^*(x) \ge c x$, then   $D_w(y) \le Cy$ and both \eqref{aintegral} and \eqref{aint-1} are satisfied. So it is sufficient to consider the case when $\lim w^*(\e_n)/\e_n = 0$ for some sequence $\e_n\to 0^+$. 

Denoting $y=w^*(x)$, so $x=D_w(y)$, and integrating by parts, we get
\[
\int_{\e_n}^\e x^{-2} w^*(x) dx = - \int_{\e_n}^\e w^*(x) d(1/x) = - w^*(x)/x \Bigm|_{\e_n}^\e + \int_{\delta_n}^\delta x^{-1} dy, 
\] 
where $\delta = w^*(\e)$, $\delta_n=w^*(\e_n)$. Taking limit as $n\to\infty$ we can see that the conditions \eqref{aintegral} and \eqref{aint-1} are equivalent. 
\end{proof}

\begin{rem*}
Condition \eqref{aintegral} is satisfied if for small $x$, $w^*(x)\ge x$, or if $w^*(x) \ge c x \ln^{-p}(1/ x)$, $p<1$, or even if   
\[
w^*(x) \ge c x/\bigl[ \,(\ln 1/x) (\ln\ln 1/x) \ldots (\underbrace{\ln\ln\ldots \ln}_{m \text{ times}} 1/x)
(\underbrace{\ln\ln\ldots \ln}_{m+1 \text{ times}} 1/x)^{p}\,\bigr]
\]
($p<1$). 

Similarly, \eqref{aint-1} holds if for $p<1$ and $t\to\infty$
\[
|\{x\in I: 1/w(x)> t\}| \le C t^{-1} (\ln t) (\ln\ln t) \ldots (\underbrace{\ln\ln\ldots \ln}_{m \text{ times}} t)
(\underbrace{\ln\ln\ldots \ln}_{m+1 \text{ times}} t)^{p}. 
\]
\end{rem*}

\begin{proof}[Proof of Lemma \ref{eqinfty}]
Since $(T_\mu)_\e^* = -(T_\nu)_\e$, the operators $(T_\nu)_\e: L^2(\nu) \to L^2(\mu)$ are uniformly bounded, and therefore they are uniformly bounded as operators $L^2(\nu)\to L^2(w)$.  Therefore we can pick  subsequence $(T_\nu)_{\e_k}$, $\e_k \to 0^+$ which converges in the weak operator topology of $B(L^2(\nu), L^2(w))$. 

Since for any $f\in L^2(\nu)$ the Cauchy integral $Kf\nu$ exists a.e.~with 
respect to Lebesgue measure%
\footnote{It is not difficult to show that under assumptions of the lemma $\int \frac{|f(t)| d\nu(t)}{1+|t|}<\infty$, but one does not need to show that, because in the proof it is sufficient to consider only compactly supported $f$.},
Lemma \ref{wkeqae} implies that the corresponding weak limit is the operator $f\mapsto K f\nu$. 

Since this operator is clearly bounded, applying it to $f=\chi\ci{ I}$, we get 
\[
\int_I |K\nu_1|^2 w(x) d x \le \int_\R |K\nu_1|^2 w(x) dx \le C \|\chi\ci{ I}\|^2\ci{L^2(\nu)} = C\nu(\chi\ci{I})  ,  
\]
where $d\nu_1 = \chi\ci{I} d\nu$. 

Using the distribution function we get that 
\[
\int_I |K\nu_1|^2 w(x) d x = \int_0^\infty 2t \int_{\{|K\nu_1|>t\}\cap I} w(x) \,dx dt.
\]

Let us assume that $\nu$ has a nontrivial singular part on $ I$, i.e.~that $\nu_1$ has a nontrivial singular part.  By Corollary \ref{APOL-t}, we have $|\{|K\nu_1|>t\}\cap I|\ge C/ t>0$ for all sufficiently large $t$ ($t\ge A$ for some $A>0$).

Let $L= | {\{|K\nu_1|>t\}\cap I |}$. Since $C/t \le L$ for $t\ge A$,  we have 
\begin{align*}
\int_0^{C/t} w^*(x) dx
\le
\int_0^{L} w^*(x) dx
\le
\int_{\{|K\nu_1|>t\}\cap I} w(x) \,dx.
\end{align*}
Multiplying this inequality by $2t$ and integrating we get
\begin{align}
\label{4.2}
\int_A^\infty 2t  \int_0^{C/t} w^*(x) dx dt  
& \le
 \int_A^\infty  2t\int_{\{|K\nu_1|>t\}\cap I} w(x) \,dx dt
\\
\notag
& \le 
\int_0^\infty 2t \int_{\{|K\nu_1|>t\}\cap I} w(x) \,dx dt
\\  \notag
& = \int_I    |K\nu_1|^2 w(x) d x
\le C\nu(\chi\ci{\clos I}) 
<\infty.
\end{align}
Using Tonelli's theorem to change the order of integration, we can write the left side as
\begin{equation}
\label{4.3}
\int_A^\infty 2t  \int_0^{C/t} w^*(x) dx dt   = \int_0^{C/A} [(C/x)^2-A^2]w^*(x)dx.
\end{equation}
Clearly $\int_0^{C/A} w^*(x) dx <\infty$. So combining \eqref{4.2} and \eqref{4.3}, we get that if the measure $\nu_1$ has a non-trivial singular part, then 
\[
\int_0^\e x^{-2} w^*(x) dx < \infty, 
\]
where $\e= C/A$. 
\end{proof}

Let $A_\alpha = A +\alpha (\fdot, \f)\f$ be the family of rank one perturbations of the operator $A$ as described in Section \ref{setup}, and let $\mu_\alpha$ be their spectral measures (corresponding to $\f$), $\mu=\mu_0$ being the spectral measure of $A$.  

Theorem below is an immediate corollary of Lemma \ref{eqinfty}.

\begin{theo}
\label{mainTM}
Let $d\mu= wdt + d\mu\ti s$ be the Lebesgue decomposition of the spectral measure $\mu=\mu^\f$.

If for a bounded closed interval $I$ the distribution function $D_w= D_{w, I}$ satisfies \eqref{aint-1} (equivalently, its inverse $w^*$ satisfies \eqref{aintegral}), then 
 for all $\alpha\in\R\setminus\{0\}$ operator $A_\alpha$ has empty singular spectrum on $I$.
\end{theo}

Let us state a similar result that incorporates the averages of the spectral measures into the hypothesis.

\begin{theo}[Averaged condition]\label{averaged}
For a finite Borel measure $\sigma$ on $\R$, define the average spectral measure $\tau = \int\mu_\beta\,d\sigma(\beta)$, and let $\tau = wdt + \tau\ti s$ be its Lebesgue decomposition. 

Let $E:=\{\alpha\in\R\,:\,\int\frac{d\sigma(\beta)}{|\alpha-\beta|^2}<\infty\}$. 

If for a bounded closed interval  $I$   the distribution function $D_w= D_{w, I}$ satisfies \eqref{aint-1}, 
then for all $\alpha\in E$ (in particular, for all $\alpha$ outside of the closed support of $\sigma$) operator $A_\alpha$  has empty singular spectrum on $I$.
\end{theo}

\begin{proof}
To apply Lemma \ref{eqinfty}, we need to show that for each $\alpha\in E$  the operators $T_\e= (T_\tau)_\e:L^2(\tau)\to L^2(\mu_\alpha)$ are uniformly bounded. 
Take $f\in L^2(\tau)$ and $g\in L^2(\mu_\alpha)$ and estimate
\begin{align*}
\left|((T_\tau)_\e f, g)\ci{L^2(\mu_\alpha)}\right|
& =
\left|\iint \frac{f(t)\overline{g(s)}}{s-t+i\e}\,d\tau(t)d\mu_\alpha(s)\right|\\
&=
\left|\iiint \frac{f(t)\overline{g(s)}}{s-t+i\e}\,d\mu_\beta(t)d\mu_\alpha(s)d\sigma(\beta)\right|\\
&\le
\int \|f\|\ci{L^2(\mu_\beta)} \|g\|\ci{L^2(\mu_\alpha)} |\alpha-\beta|^{-1} d\sigma(\beta)\\
&\le
\|g\|\ci{L^2(\mu_\alpha)}
\left(\int\frac{d\sigma(\beta)}{|\alpha-\beta|^2}\right)^{1/2}
\left(\int \|f\|^2\ci{L^2(\mu_\beta)} d\sigma(\beta)\right)^{1/2}.
\end{align*}
Note that the last factor on the right hand side is equal to $\|f\|\ci{L^2(\mu)} \sigma(\R)$ and recall that $\int\frac{d\sigma(\beta)}{|\alpha-\beta|^2}\le C<\infty$.
\end{proof}

\section{Some examples}\label{REXA}
Theorem \ref{mainTM} can be used to construct examples of rank one perturbations with weird behavior. Consider first an abstract situation. 

\subsection{Friedrichs model} 
\label{s5.1}
Let $\mu$ be a finite Borel measure supported on a finite closed interval $I$, and let $d\mu = wdt + d\mu\ti{s}$ be its Lebesgue decomposition. Let  operator $A$ be the multiplication $M_t$ by the independent variable $t$ in $L^2(\mu)$.

Let the density $w$ on the interval $I$ satisfy condition \eqref{aint-1}. Assume also that the closed support of $\mu\ti{s}$ coincides with $I$. Then, first of all, by Theorem \ref{mainTM}, the perturbed operators $A_\alpha := A+ \alpha (\fdot, \ID)\ID$ have no singular spectrum on $I$ for all $\alpha\ne0$. Of course, an eigenvalue outside of $I $ can appear. 

Second, the density $w_\alpha$ of the spectral measure $\mu_\alpha$ of $A_\alpha$ is highly irregular: It fails to satisfy condition \eqref{aint-1} on any subinterval of $I$. 

Indeed, one can write $A= A_\alpha -\alpha (\fdot, \ID)\ID$. Since the close support of the singular part of $\mu\ti{s}$ is the whole interval $I$, condition \eqref{aint-1} must fail for all its subintervals. 

If we consider perturbations $A_{\alpha_0} + \alpha (\fdot, \ID)\ID$ of the operator $A_{\alpha_0}$, $\alpha_0\ne0$, then we get a family of rank one perturbations for which the singular spectrum appears at exactly one value of the parameter $\alpha$ ($\alpha = - \alpha_0$). 

If the condition \eqref{aint-1} holds for any subinterval $J\Subset I$, then we can conclude that all perturbations $A_\alpha$ have no singular spectrum in the interior of $I$ (atoms  can appear at the endpoints). 

\subsection{Jacobi matrices} The same reasoning as above in Section \ref{s5.1} can be applied  to  Jacobi matrices. By a \emph{Jacobi matrix} we refer to a semi-infinite tridiagonal matrix of the form
\[
T:=
\begin{pmatrix}
b_1&a_1&0&\hdots&\hdots&\hdots\\
a_1&b_2&a_2&0&\hdots&\hdots\\
0&a_2&b_3&a_3&0&\hdots\\
\vdots&\ddots&\ddots&\ddots&\ddots&\ddots
\end{pmatrix}
\]
where $a_n>0$, $b_n\in\R$ for all $n\in\N$. The \emph{free Jacobi matrix} $T_0$, is the Jacobi matrix with $b_n=0$ and $a_n=1$ for all $n\in\N$. We assume also that $\sup_n |a_n|+|b_n|<\infty$, so a Jacobi matrix can be viewed as a bounded operator on $\ell^2=\ell^2(\N)$ (the \emph{Jacobi operator}).

As it is well-known, see e.g.~\cite{DEIFT}, that there is a one-to-one correspondence between compactly supported Borel measures on $\R$ satisfying the normalization condition $\mu(\R)=1$  and bounded Jacobi operators. Namely, any such measure is the spectral measure (corresponding to the cyclic vector $e_1$) of the corresponding  Jacobi matrix; here $\{e_n\}_{n=1}^\infty$ is the standard basis in $\ell^2$. 

So all that was said above in Section \ref{s5.1} about perturbations of multiplication operator can be trivially said about perturbations $T_\alpha$ of a Jacobi matrix $T$, $T_\alpha = T + \alpha(\fdot, e_1)e_1$; note that $T_\alpha$ is obtained from $T$ by replacing the entry $b_1$ in the Jacobi matrix by $b_1+\alpha$. 

What is more interesting, the same can be said about Jacobi matrices that are Hilbert--Schmidt perturbations  of the free Jacobi matrix, i.e.~about Jacobi matrices such that 
\[
\sum_{n=1}^\infty(a_n-1)^2+b_n^2<\infty.
\]
In \cite{KILLIPSIM} a complete description of spectral measures of such matrices was obtained.

\begin{theo}\label{KILLIPSIM} (Killip--Simon \cite{KILLIPSIM})
Let $J$ be a Jacobi matrix and $\mu$ be the corresponding spectral measure (corresponding to the vector $e_1$).

Operator $T-T_0$ is Hilbert--Schmidt if and only if all four conditions hold
\begin{itemize}
\item[(1)] Blumenthal--Weyl: $\supp\,d\mu=[-2,2]\cup\{\lambda^+_j\}\cup\{\lambda^-_j\},
$  where $\{\lambda^\pm_j\}$ denote the
sequences of eigenvalues of $J$ in $\R\backslash [-2,2]$  and $\lambda^+_1>\lambda^+_2>\hdots>2$ and $\lambda^-_1<\lambda^-_2<\hdots<-2$,
\item[(2)] Lieb--Thirring: $\sum_j(\lambda^+_j-2)^{3/2}+\sum_j(\lambda^-_j+2)^{3/2}<\infty$,
\item[(3)] Quasi-Szeg\"o: $\int_{-2}^2(4-t^2)^{1/2}\log (w(t))\,\dd t > - \infty$, where $w$ is the density function of $\mu$, i.e.~ $d\mu=w dt+d \mu\ti{s}$,
\item[(4)] Normalization: $\mu(\R)=1$.
\end{itemize}
\end{theo}

It is easy to see, that one can construct a measure $\mu$ satisfying all four conditions of Theorem \ref{KILLIPSIM}, and such that  condition \eqref{aint-1} is satisfied for the interval $[-2, 2]$. So, the reasoning of the previous subsection applies to this case and the perturbations $T_\alpha$ of $T$ have no singular spectrum on $[-2, 2]$. Considering perturbations of $T_{\alpha_0}$, $\alpha_0\ne0$, one comes up with the example of a family of rank one perturbations $T_{\alpha_0} + \alpha (\fdot, e_1)e_1$ such that the singular spectrum on $\sigma\ti{ess}(T)$ appears only for one value of $\alpha$. Note, operator $T_{\alpha_0}$ is a Hilbert--Schmidt perturbation of the free Jacobi matrix.

\subsection{Schr\"odinger operators} The same idea as in Section \ref{s5.1} can be applied to  \emph{(half-line) Schr\"odinger operators} $H:=-\frac{d^2 }{dx^2}+V$ with $L^2$ potentials ($V\in L^2(\R_+)$) on $L^2(\R_+)$, $\R_+:=(0,\infty)$.

Let us recall that for a formal differential operator $H=H_V = -\frac{d^2 }{dx^2}+V$ on $\R_+$, $V\in L^2(\R_+)$ one can define a family of self-adjoint operators $H_\vartheta$ on $L^2(\R_+)$ defined by the boundary condition at $0$; these operators differ by their respective domains,  
\begin{align*}
D(H_\vartheta)=
&\{u\in L^2(\R^+):u,u'\text{ are locally absolutely continuous, }\\
&\,\,u(0)\cos(\vartheta)+u'(0)\sin(\vartheta)=0\quad \text{for }0\le\vartheta<\pi, \quad H_\vartheta u\in L^2(\R^+)\}.
\end{align*}
Note that $\vartheta=0$ corresponds to the Dirichlet boundary condition and $\vartheta=\pi/2$ corresponds to the Neumann boundary condition. Recall that if $V\in L^2(\R^+)$, then $H$ is limit point, see e.g.~\cite{Naim}, meaning that Dirichlet boundary conditions (and so the boundary conditions for $H_\vartheta$)  define a self-adjoint operator.

A recent theorem of Killip and Simon \cite{SKILLIPSIM}
gives a complete description of spectral measures of Schr\"odinger operators with $L^2$ potentials (with Dirichlet boundary condition).  Without stating Killip--Simon theorem here, we will only mention that it is not hard to construct a measure $\mu$ satisfying the conditions of this theorem and such that its weight $w$ satisfies  condition \eqref{aint-1}  for all bounded intervals $I\subset [0, \infty)$. Moreover, it is not hard to show that the singular part of $\mu$ can be essentially arbitrary, i.e.~given a singular Radon measure $\tau$ on $\R_+$ one can find $\mu$ satisfying the conditions of the Killip--Simon theorem and such that the singular part $\mu\ti{s}$ of $\mu$ is mutually absolutely continuous with $\tau$ (and the density $w$ satisfies \eqref{aint-1}). 

It is well-known that the Schr\"odinger operators with mixed boundary conditions are viewed as self-adjoint  rank one perturbations of the Schr\"odinger $H_0$ operator with Dirichlet boundary conditions.

Unfortunately, our results cannot be applied directly, because to get from the $H_0$ to $H_\vartheta$  the perturbation should be written as $H_0 + \alpha(\vartheta) (\fdot, \delta_0')\delta_0'$, where $\delta_0'$ is the derivative of delta function at zero. The spectral measure, which is traditionally defined via the Weil $M$-function, is also the spectral measure with respect to $\delta'_0$.  But vector $\delta'_0$ is not in $\cH_{-1}(H_0)$, one can only prove that it is in $\cH_{-2}(H_0)$. However, there is a simple workaround: one just needs to consider resolvents. 

Namely, fix $\vartheta$ and consider $\la<0$ which is not an eigenvalue of $H_0$ or $H_\vartheta$. Then the difference of the resolvents can be formally written as 
\begin{align}\label{schrres}
&(H_\vartheta-\lambda\OID)^{-1} =(H_0-\lambda\OID)^{-1} +\wt\alpha(\vartheta)(\fdot, (H_0-\lambda\OID)^{-1}\delta'_0)(H_0-\lambda\OID)^{-1}\delta'_0
\end{align}
where $\wt\alpha(\vartheta) = \alpha(\vartheta) /\left[1 + \alpha(\vartheta) \left((A-\la \OID)^{-1} \delta'_0, \delta'_0 \right) \right] $. The fact that the difference of resolvents is a rank one operator follows from the standard theory of differential operators, and knowing the resolvent one defines the operator. Thus one can avoid rather complicated construction of rank one perturbations with $\f\in \cH_{-2}$. This case is described, for example, in \cite{AK1}.

The 
spectral measure $\nu$ of the resolvent $(H_0 -\la \OID)^{-1}$ can be easily computed from the spectral measure $\mu$ of $H_0$, and it is clear that if the density of $\mu$ satisfies the assumption \eqref{aint-1} on any subinterval $I\subset [0, \infty)$, then the density of $\nu$ satisfies the same condition \eqref{aint-1} for any subinterval of $(0, -1/\la]$.

So, one can apply Theorem \ref{mainTM}  to the resolvents. In doing so one can obtain all the phenomena discussed in Section \ref{s5.1}. For example, one can get $H=-\frac{d^2 }{dx^2}+V$, $V\in L^2(\R_+)$ such that $H_0$ (Dirichlet boundary conditions) has dense in $\R_+$ singular spectrum, but for all other boundary conditions the operators $H_\vartheta$ have no singular spectrum on $\R_+$. And the density of the spectral measure of $H_\vartheta$ will exhibit some weird behavior: In particular, it will not satisfy condition \eqref{aint-1} on any bounded subinterval of $\R_+$.

\providecommand{\bysame}{\leavevmode\hbox to3em{\hrulefill}\thinspace}
\providecommand{\MR}{\relax\ifhmode\unskip\space\fi MR }
\providecommand{\MRhref}[2]{%
  \href{http://www.ams.org/mathscinet-getitem?mr=#1}{#2}
}
\providecommand{\href}[2]{#2}


\begin{thebibliography}{10}

\bibitem{AK1}
S.~Albeverio and P.~Kurasov, \emph{{Rank one perturbations of not semibounded
  operators}}, Integral Equations and Operator Theory \textbf{27} (1997),
  no.~4, 379--400. 
  
  

\bibitem{birst}
M.S. Birman and M.Z. Solomjak, \emph{{Spectral theory of self-adjoint operators
  in Hilbert space}}, 1986.

\bibitem{DEIFT}
P.~Deift, \emph{{Orthogonal Polynomials and Random Matrices: A Riemann-Hilbert
  Approach}}, American Mathematical Society, 2000.

\bibitem{GOLU}
M.~G. Goluzina, \emph{On the multiplication and division of
  {C}auchy-{S}tieltjes-type integrals}, Vestnik Leningrad. Univ. Mat. Mekh.
  Astronom. (1981), 8--15, 124. \MR{MR664739 (84a:30074)}


\bibitem{SKILLIPSIM}
R.~Killip and B.~Simon, \emph{{Sum rules and spectral measures of
  {S}chr{\"o}dinger operators with {$L^2$} potentials}}, to appear in Ann. of
  Math.

\bibitem{KILLIPSIM}
R.~Killip and B.~Simon, \emph{{Sum rules for Jacobi matrices and their
  applications to spectral theory}}, Ann. of Math. \textbf{158} (2003), no.~1,
  253--321.

\bibitem{KL}
P.~Kurasov, \emph{{Singular and supersingular perturbations: Hilbert space
  methods}}, Spectral Theory of Schr{\"o}dinger Operators (2004).

\bibitem{Naim}
M.A. Naimark, \emph{{Linear differential operators}}, New York: Ungar, 1968.

\bibitem{NTV}
F.~Nazarov, S.~Treil, and A.~Volberg, \emph{{The Tb-theorem on non-homogeneous
  spaces}}, Acta Mathematica \textbf{190} (2003), no.~2, 151--239.

\bibitem{NONTAN}
A.~G. Poltoratski{\u\i}, \emph{Boundary behavior of pseudocontinuable
  functions}, Algebra i Analiz \textbf{5} (1993), no.~2, 189--210, engl.
  translation in \emph{St. Petersburg Math. J.}, 5(2): 389--406, 1994.
  \MR{MR1223178 (94k:30090)}

\bibitem{RS}
M.~Reed and B.~Simon, \emph{{Methods of Modern Mathematical Physics II: Fourier
  Analysis}}, Academic Press, 1975, Reprinted by Elsevier, ISBN 0-12-585002-6.

\bibitem{SIMREV}
B.~Simon, \emph{{Spectral analysis of rank one perturbations and
  applications}}, Mathematical Quantum Theory I: Field Theory and Many-Body
  Theory (1994).

\end{thebibliography}
\end{document}